\documentclass[11pt,twoside]{article}
\usepackage{amsmath,amssymb}
\usepackage[english]{babel} %francais, polish, spanish, ...
\usepackage{graphicx} %%For loading graphic files
\usepackage{amsmath}
\usepackage{amssymb}
\usepackage{verbatim}
\usepackage{tabu}
\usepackage{multirow}
\usepackage{enumerate}
\usepackage{amsfonts}
\usepackage{cite}
\usepackage{color}
\usepackage[overload]{empheq}
\usepackage{cleveref}

\newtheorem{proposition}{Proposition}
\newtheorem{theorem}[proposition]{Theorem}
\newtheorem{lemma}[proposition]{Lemma}
\newtheorem{corollary}[proposition]{Corollary}
\newtheorem{definition}[proposition]{Definition}
\newtheorem{remark}[proposition]{Remark}

\newcommand\restr[2]{{% we make the whole thing an ordinary symbol
  \left.\kern-\nulldelimiterspace % automatically resize the bar with \right
  #1 % the function
  \vphantom{\big|} % pretend it's a little taller at normal size
  \right|_{#2} % this is the delimiter
  }}

\newenvironment{proof}{{\noindent \it Proof.}}{\hfill $\fbox{}$ \vspace*{5mm}}

\newcommand{\X}{{\cal X}}
\newcommand{\Y}{{\cal Y}}

\newcommand{\N}{\mathbb{N}}
\newcommand{\R}{\mathbb{R}}

\newcommand{\BE}{\begin{equation}}
\newcommand{\EE}{\end{equation}}

\numberwithin{equation}{section}
\date{}

\begin{document}
\title{Iterated fractional Tikhonov regularization}
\author{Davide Bianchi$^*$ \and 
Alessandro Buccini$^*$ \and 
Marco Donatelli$^*$ \and 
Stefano Serra-Capizzano\footnote{Dipartimento di Scienza e Alta Tecnologia, 
                                 Universit\`a dell'Insubria,
                                 22100 Como, Italy}}

\maketitle

\begin{abstract}
Fractional Tikhonov regularization methods have been recently proposed to reduce the oversmoothing 
property of the Tikhonov regularization in standard form, in order to preserve the details of the approximated solution. 
Their regularization and convergence properties have been previously investigated showing that they are of optimal order.
This paper provides saturation and converse results on their convergence rates. 
Using the same iterative refinement strategy of iterated Tikhonov regularization, new iterated fractional Tikhonov regularization methods are introduced. We show that these iterated methods are of optimal order and overcome the previous saturation results. Furthermore, nonstationary iterated fractional Tikhonov regularization methods are investigated, establishing their convergence rate under general conditions on the iteration parameters. Numerical results confirm the effectiveness of the proposed regularization iterations.      
\end{abstract}

%-------------------------------------------------------------
\section{Introduction}
%-------------------------------------------------------------
We consider linear operator equations of the form
\begin{equation}\label{eq:Kxy}
Kx \,=\, y\,,
\end{equation}
where $K:\X\to\Y$ is a compact linear operator between Hilbert
spaces $\X$ and $\Y$. We assume $y$ to be attainable, i.e., that
problem~\eqref{eq:Kxy} has a solution  $x^\dag=K^\dag y$ of
minimal norm. Here $K^\dag$ denotes the (Moore-Penrose)
generalized inverse operator of $K$, which is unbounded when $K$
is compact, with infinite dimensional range. Hence problem~\eqref{eq:Kxy} is ill-posed and has to be
regularized in order to compute a numerical solution; see 
\cite{engl1996regularization}.
%In general terms, the problem~\eqref{eq:Kxy} is approximated by
%a family of neighbouring well-posed problems \cite{engl1996regularization}.

We want to approximate the solution $x^\dag$ of the equation~\eqref{eq:Kxy}, 
when only an approximation $y^\delta$ of $y$ is available with
\begin{equation}\label{eq:delta}
\|y^\delta - y \|\leq \delta,
\end{equation}
where $\delta$ is called the noise level. Since $K^\dag y^\delta$ is not a good
approximation of $x^\dag$, we approximate $x^\dag$ with
$x_\alpha ^\delta := R_\alpha y^\delta$
where $\{R_\alpha\}$ is a family of continuous operators depending on a parameter
$\alpha$ that will be defined later.
A classical example is the Tikhonov regularization defined by
$R_\alpha=(K^*K+\alpha I)^{-1}K^*$,
where $I$ denotes the identity and $K^*$ the adjoint of $K$, cf. \cite{groetsch1984}.

Using the singular values expansion of $K$, filter based
regularization methods are defined in terms of filters of the
singular values, cf. Proposition~\ref{prop:filter}. This is a
useful tool for the analysis of regularization techniques
\cite{hansen1998}, both for direct and iterative regularization
methods
\cite{hankehansen1993,hno2006}.
Furthermore, new regularization methods can be defined
investigating new classes of filters. For instance, one of the
contributes in \cite{klann2008regularization} is the proposal and
the analysis of the fractional Tikhonov method. The authors obtain
a new class of filtering regularization methods adding an
exponent, depending on a parameter, to the filter of the standard
Tikhonov method. They provide a detailed analysis  of the
filtering properties and the optimality order of the method in
terms of such further parameter. A different generalization of the
Tikhonov method has been recently proposed in
\cite{hochstenbach2011fractional} with a detailed filtering analysis.
Both generalizations are called ``fractional Tikhonov regularization'' in the literature and 
they are compared in \cite{Gert2015fractional}, where the optimality order of the method in
\cite{hochstenbach2011fractional} is provided as well. To distinguish the two proposals in 
\cite{klann2008regularization} and \cite{hochstenbach2011fractional}, we will refer in the following as 
``fractional Tikhonov regularization'' and ``weighted Tikhonov regularization'', respectively.
These variants of the Tikhonov method have
been introduced to compute good approximations of non-smooth
solutions, since it is well known that the Tikhonov method
provides over-smoothed solutions. 
%Moreover, they have shown to be very effective when the observed data
%are contaminated by low frequency noise \cite{Gert2015fractional}.

In this paper, we firstly provide a saturation result similar to the well-known saturation result
for Tikhonov regularization \cite{engl1996regularization}: let $R(K)$ be the range of $K$ and let $Q$ be the orthogonal projector onto $\overline{R(K)}$, if
\[
\sup \left\{ \| x_{\alpha}^\delta - x^\dagger \| :\, \|Q( y - y^\delta) \| \leq \delta  \right\} = o(\delta^{\frac{2}{3}}),
\]
then $x^\dagger=0$, as long as $\overline{R(K)}$ is not closed.
Such result motivated us to introduce the iterated version of fractional and weighted Tikhonov in the same spirit of the 
iterated Tikhonov method. We prove that those iterated methods can overcome the previous saturation results.
Afterwards, inspired by the works \cite{brill1987iterative,hanke1998nonstationary} we introduce the nonstationary variants of our
iterated methods. Differently from the nonstationary iterated Tikhonov, we have two nonstationary sequences of parameters.
In the noise free case, we give sufficient conditions on these sequences to guarantee the convergence providing also the corresponding convergence rates. In the noise case, we show the stability of the proposed iterative schemes proving that they are regularization methods. Finally, few selected examples confirm the previous theoretical analysis, showing that a proper choice of the nonstationary sequences of parameters can provide better restorations compared to the classical iterated Tikhonov with a geometric sequence of regularizzation parameter according to \cite{hanke1998nonstationary}. 

The paper is organized as follows. Section~\ref{sec:basic} 
recalls the basic definition of filter based regularization methods
and of optimal order of a regularization method. 
Fractional Tikhonov regularization with optimal order and converse results are studied in Section~\ref{sec:fract}. 
Section~\ref{sec:saturation} is devoted to saturation results for both variants of fractional Tikhonov regularization.
New iterated fractional Tikhonov regularization methods are introduced in Section~\ref{sec:statiter}, 
where the analysis of their convergence rate shows that their are able to overcome the previous saturation results.
A nonstationary iterated weighted Tikhonov regularization is investigated in detail in Section~\ref{sec:NSIWT},
while a similar nonstationary iterated fractional Tikhonov regularization is discussed in Section~\ref{sec:NSIFT}.
Finally, some numerical examples are reported in Section~\ref{sec:numerical}.

%------------------------------------------------------------------------------------------------
\section{Preliminaries}\label{sec:basic}
%------------------------------------------------------------------------------------------------

As described in the Introduction, we consider a compact linear
operator $K : \X \to \Y$ between Hilbert spaces $\X$ and $\Y$ (over the field $\R$ or $\mathbb{C}$) with given inner products $\langle \cdot , \cdot \rangle_{\X}$ and $\langle \cdot , \cdot \rangle_{\Y}$, respectively. Hereafter we will omit the subscript for the inner product as it will be clear in the context. If
$K^* : \Y \to \X$ denotes the adjoint of $K$ (i.e., $\langle Kx, y
\rangle = \langle x, K^* y \rangle$), then we indicate with $\left
( \sigma_n; v_n, u_n \right )_{n \in \mathbb{N}}$ the singular
value expansion (s.v.e.) of $K$, where $\{ v_n \}_{n\in
\mathbb{N}}$ and $\{ u_n \}_{n\in \mathbb{N}}$ are a complete
orthonormal system of eigenvectors for $K^* K$ and $KK^*$,
respectively, and $\sigma_n > 0$ are written in decreasing order,
with $0$ being the only accumulating point for the sequence $\{ \sigma_n
\}_{n\in \mathbb{N}}$. If $\X$ is not finite dimensional, then $0 \in \sigma(K^*K)$, the spectrum of $K^*K$, namely $\sigma(K^*K) = \{0\} \cup \bigcup_{n=1}^\infty \{\sigma_n ^2 \}$. Finally, $\sigma(K)$ denotes the closure of
$\bigcup_{n=1}^\infty\{ \sigma_n \}$, i.e., $ \sigma(K)= \{ 0 \}
\cup \bigcup_{n=1}^\infty \{ \sigma_n \}$.

Let now $\{ E_{\sigma^2} \}_{\sigma^2 \in \sigma(K^* K)}$ be the spectral
decomposition of  the self-adjoint operator $K^* K$. Then from
well-known facts from functional analysis \cite{rudin1991functional}
we can write $f(K^* K) := \int f(\sigma ^2) dE_{\sigma^2}$, where $f : \sigma (K^*K) \subset \R \to \mathbb{C}$ is a bounded Borel measurable function and
$\langle E x_1 , x_2 \rangle$ is a regular complex Borel measure for every $x_1, x_2
\in \X$. The following equalities hold
\begin{align}
%&Kv_n = \sigma_n u_n,\\
%&K^* u_n = \sigma_n v_n, \label{eq:1.7}\\
&Kx = \sum_{m=1}^{+\infty} \sigma_m \langle x, v_m \rangle u_m, \qquad x \in \X, \label{eq:1.1} \\
&K^* y = \sum_{m=1}^{+\infty} \sigma_m \langle y, u_m \rangle v_m, \qquad y \in \Y, \label{eq:1.2} \\
&f(K^* K) x := \int_{\sigma(K^*K)} f(\sigma^2) dE_{\sigma^2} x  = \sum_{m=1}^\infty f(\sigma_m ^2)\langle x,v_m \rangle v_m,\\
&\langle f(K^* K) x_1 , x_2\rangle = \int_{\sigma(K^*K)} f(\sigma^2) d\langle E_{\sigma^2} x_1 , x_2 \rangle = \sum_{m=1}^\infty f(\sigma_m ^2) \overline{\langle y, v_m \rangle} \langle x, v_m \rangle ,\\
%&\int_{\mathbb{R}} f(\lambda) dE_\lambda x  = \lim_{\epsilon \to 0} \int_{0}^{\| K \|^2 +\epsilon} f(\lambda) dE_\lambda x,
&\| f(K^*K) \| \leq \sup \{ |f(\sigma ^2)| : \sigma^2 \in \sigma(K^*K)  \}, \label{eq:1.2.1}
\end{align}
where the series \eqref{eq:1.1} and \eqref{eq:1.2} converge in the
$L^2$ norms induced by the scalar products of $\X$ and $\Y$,
respectively. If $f$ is a continuous function on $\sigma(K^*K)$ then equality holds in \eqref{eq:1.2.1}.

\begin{definition}%[Generalized Inverse]
We define the \emph{generalized inverse} $K^\dagger : \mathcal{D}(K^\dagger) \subseteq \Y \to \X$ of a compact linear operator $K : \X \to \Y$ as
\begin{equation}\label{eq:xdag}
K^\dagger y=  \sum_{m:\, \sigma_m >0} \sigma_m ^{-1} \langle y, u_m \rangle v_m, \qquad y \in \mathcal{D}(K^\dagger),
\end{equation}
where
$$\mathcal{D}(K^\dagger) = \left \{ y \in \Y : \sum_{m:\, \sigma_m >0} \sigma_m ^{-2} | \langle y, u_m \rangle | ^2 < \infty \right \}.
$$
\end{definition}

With respect to problem~\eqref{eq:Kxy}, we consider the case where
only an approximation $y^\delta$ of $y$ satisfying the condition
\eqref{eq:delta} is available. Therefore $x^\dagger = K^\dagger
y$, $ y \in \mathcal{D}(K^\dagger)$, cannot be approximated by
$K^\dag y^\delta$, due to the unboundedness of $K^\dag$, and hence
in practice the problem~\eqref{eq:Kxy} is approximated by a family
of neighbouring well-posed problems \cite{engl1996regularization}.

\begin{definition}%[Regularization method]
\label{definition:1.1}
By a \emph{regularization method} for $K^\dagger$ %(see \cite{engl1996regularization})
we call any family of operators
$$
\{ R_\alpha \}_{\alpha \in (0, \alpha_0)} : \Y \to \X, \qquad \alpha_0 \in (0, +\infty],
$$
with the following properties:
\begin{itemize}
\item[(i)] $R_\alpha : \Y \to \X$ is a bounded operator for every
$\alpha$. \item[(ii)] For every $y \in \mathcal{D}(K^\dagger)$
there exists  a mapping (\textit{rule choice}) $\alpha :
\mathbb{R}_+ \times \Y \to (0, \alpha_0)\in\R$, $\alpha = \alpha
(\delta, y^\delta)$, such that
$$
\limsup_{\delta \to 0} \left \{ \alpha (\delta, y^\delta) : y^\delta \in \Y, \| y - y^\delta \| \leq \delta \right\} =0,
$$
and
$$
\limsup_{\delta \to 0} \left \{ \| R_{\alpha(\delta, y^\delta)} y^\delta - K^\dagger y \| : y^\delta \in \Y, \| y - y^\delta \| \leq \delta \right \}=0.
$$
\end{itemize}
\end{definition}

%\begin{remark}
Throughout this paper $c$ is a constant which can change from one
instance to the next. For the sake of clarity, if more than one
constant will appear in the same line or equation we will
distinguish them by means of a subscript.
%\end{remark}

\begin{proposition}%[Filter-based regularization method]
\label{prop:filter}
Let $K : \X \to \Y$ be a compact linear operator and $K^\dagger$ its generalized inverse. Let $R_\alpha : \Y \to \X$ be a family of operators defined for every $\alpha \in (0, \alpha_0)$ as
\begin{equation}\label{eq:def_filter}
R_\alpha y := \sum_{m:\, \sigma_m >0} F_\alpha (\sigma_m) \sigma_m ^{-1} \langle y, u_m \rangle v_m,
\end{equation}
where $F_\alpha : [0, \sigma_1] \supset \sigma(K) \to \mathbb{R}$ is a Borel function such that
\begin{subequations}\label{eq:filtercond}
\begin{eqnarray}
    %\mbox{(a) }
&& \sup_{m:\, \sigma_m >0} | F_\alpha (\sigma_m) \sigma_m ^{-1} | = c(\alpha) < \infty, %\qquad \mbox{for} \alpha\in (0, \alpha_0),
\label{eq:1.3}\\
    %\mbox{(b) }
&&  | F_\alpha (\sigma_m)| \le c <
\infty,  \qquad \mbox{where }c\mbox{ does not depend on }
(\alpha,m), %\ \mbox{nor } \sigma_m, , %\in (0, \alpha_0),
\label{eq:1.4} \\
    %\mbox{(c) }
&&  \lim_{\alpha \to 0} F_\alpha (\sigma_m)  = 1 \, \mbox{
point-wise in } \sigma_m. \label{eq:1.5}
\end{eqnarray}
\end{subequations}
Then $R_\alpha$ is a regularization method, with $\| R_\alpha \| = c(\alpha)$, and it is called \emph{filter based regularization method}.
\end{proposition}
\begin{proof}
See \cite{louis65075inverse} and \cite{engl1996regularization}.
\end{proof}

For the sake of notational brevity, we fix the following notation
\begin{align}
x_\alpha & := R_\alpha y, \qquad y \in \mathcal{D}(K^\dagger), \label{eq:xalpha}\\
x_\alpha ^\delta & := R_\alpha y^\delta, \qquad y^\delta \in \Y. \label{eq:xdeltaalpha}
\end{align}
We report hereafter the definition of optimal order, under the
same a-priori assumption given in \cite{engl1996regularization}.

\begin{definition}%[Optimal order under a-priori assumption]
For every given $\nu, \rho >0$, let
$$
\X_{\nu, \rho} := \left \{ x \in \X : \exists \, \omega \in \X, \| \omega \| \leq \rho, x = (K^* K)^{\frac{\nu}{2}}\omega \right \} \subset \X.
$$
A regularization method $R_\alpha$ is called of \emph{optimal order} under the a-priori assumption $x^\dagger \in  \X_{\nu, \rho}$ \nolinebreak if
\begin{equation}\label{eq:1.6}
\Delta (\delta, \X_{\nu, \rho}, R_\alpha) \leq c \cdot \delta^{\frac{\nu}{\nu +1}} \rho ^{\frac{1}{\nu +1}},
\end{equation}
where for any general set $M \subseteq X$, $\delta >0$ and for a regularization method $R_\alpha$, we define
$$
\Delta (\delta, M, R_\alpha)  := \sup \left\{  \| x^\dagger - x_\alpha ^\delta \| : x^\dagger \in M,\, \| y - y^\delta \| \leq \delta \right\}.
$$
If $\rho$ is not known, as it will be usually the case, then we relax the definition introducing the set
$$
\X_\nu := \bigcup_{\rho >0} \X_{\nu, \rho}
$$
and saying that a regularization method $R_\alpha$ is called of \emph{optimal order} under the a-priori assumption $x^\dagger \in  \X_{\nu}$ if
\begin{equation}\label{eq:1.6.1}
\Delta (\delta, \X_{\nu}, R_\alpha) \leq c \cdot \delta^{\frac{\nu}{\nu +1}}.
\end{equation}
\end{definition}

\begin{remark}
Since we are concerned with the rate that $ \|x^\dagger - x_\alpha ^\delta \| $ converges to zero as $\delta \to 0$, the a-priori assumption $x^\dagger \in  \X_{\nu}$ is usually sufficient for the optimal order analysis, requiring that \eqref{eq:1.6.1} is satisfied.
\end{remark}

Hereafter we cite a theorem which states sufficient conditions for order optimality, when  filtering methods are employed, see \cite[Proposition 3.4.3, pag. 58]{louis65075inverse}.
\begin{theorem}{\cite{louis65075inverse}}\label{proposition:2.1}
Let $K : \X \to \Y$ be a compact linear operator, $\nu$ and $\rho >0$,
%, $(\sigma_n; u_n, v_n)_{n \in \mathbb{N}}$ its s.v.e.
and let $R_\alpha : \Y \to \X$ be a filter based regularization method.
%defined in \eqref{eq:def_filter}.
If there exists a fixed $\beta >0$ such that
%\begin{align}
%&\mbox{(i) } \sup_{0< \sigma \leq \sigma_1} | F_\alpha (\sigma) \sigma ^{-2}| \leq c \cdot \alpha^{-\beta}, \label{eq:2.1}\\
%&\mbox{(ii) } \sup_{0 \leq \sigma \leq \sigma_1} | (1 - F_\alpha (\sigma)) \sigma ^\nu | \leq c \cdot \alpha^{\beta \frac{\nu}{2}} \label{eq:2.2}
%\end{align}
\begin{subequations}\label{eq:2}
\begin{equation}
    \sup_{0< \sigma \leq \sigma_1} | F_\alpha (\sigma) \sigma ^{-1}| \leq c \cdot \alpha^{-\beta}, \label{eq:2.1}
\end{equation}
\begin{equation}
    \sup_{0 \leq \sigma \leq \sigma_1} | (1 - F_\alpha (\sigma)) \sigma ^\nu | \leq c_\nu \cdot \alpha^{\beta \nu}, \label{eq:2.2}
\end{equation}
\end{subequations}
then 
%\begin{itemize}
%\item[(i)]
 $R_\alpha$ is of optimal order, under the a-priori assumption $x^\dagger \in \X_{\nu, \rho}$, with the choice rule 
$$
\alpha = \alpha (\delta, \rho) = O \left ( \frac{\delta}{\rho} \right)^{\frac{1}{\beta (\nu +1)}}.
$$
%\item[(ii)] $R_\alpha$ is of optimal order under the a-priori assumption $x^\dagger \in \X_{\nu}$ with the choice rule $$
%\alpha = \alpha(\delta)= \delta^{\frac{1}{\beta (\nu +1)}}.
%$$
%\end{itemize}
\end{theorem}

If we are concerned just about the rate of convergence with respect to only $\delta$, the preceding theorem can be applied under the a-priori assumption $x^\dagger \in X_\nu$, fitting the proof to the latter case without any effort. On the contrary, below we present a converse result.

\begin{theorem}\label{thm:con.result}
Let $K$ be a compact linear operator with infinite dimensional range and let $R_\alpha$ be a filter based regularization method with filter function $F_\alpha : [0, \sigma_1] \supset \sigma(K) \to \mathbb{R}$. If there exist $\nu$ and $\beta>0$ such that 
\begin{equation}\label{eq:conv.result1}
\left(1- F_\alpha (\sigma) \right) \sigma^\nu \geq c \alpha^{\beta \nu} \qquad \mbox{for } \sigma \in [c' \alpha^\beta, \sigma_1]
\end{equation}
and
\begin{equation}\label{eq:conv.result2}
\| x^\dagger - x_\alpha \| = O(\alpha^{\beta \nu}),
\end{equation}
then $x^\dagger \in \X_\nu$.
\end{theorem}
\begin{proof}
By \eqref{eq:xdag} and \eqref{eq:def_filter}, it holds
\begin{align*}
\| x^\dagger - x_\alpha \|^2 &= \sum_{\sigma_m >0} \left(1 - F_\alpha(\sigma_m) \right)^2 \sigma_m ^{-2} | \langle y, u_m \rangle |^2 \\
&= \sum_{\sigma_m >0} \left(1 - F_\alpha(\sigma_m) \right)^2  | \langle x^\dagger, v_m \rangle |^2 \\
&=  \sum_{\sigma_m >0} \left[ \left(1 - F_\alpha (\sigma_m) \right)\sigma_m ^\nu \right]^2 \sigma_m ^{-2\nu} | \langle x^\dagger, v_m \rangle |^2 \\
&\geq \left(c\alpha^{\beta \nu} \right)^2 \sum_{\sigma_m \geq c' \alpha^\beta} \sigma_m ^{-2\nu} |\langle x^\dagger, v_m \rangle|^2.
\end{align*}
thanks to the assumption \eqref{eq:conv.result1}.
From \eqref{eq:conv.result2} we deduce that 
$$
\lim_{\alpha^\beta \to 0} \sum_{\sigma_m \geq c' \alpha^\beta} \sigma_m ^{-2\nu} |\langle x^\dagger, v_m \rangle|^2 < +\infty.
$$
Finally, if we define $\omega := \sum_{\sigma_m >0} \sigma^{-\nu} \langle x^\dagger, v_m \rangle v_m$, then $\omega$ is well defined and $\left(K^* K\right)^{\nu/2} \omega = x^\dagger$, i.e., $x^\dagger \in X_\nu$.
\end{proof}

%-------------------------------------------------------------------
\section{Fractional variants of Tikhonov regularization}\label{sec:fract} 
%-------------------------------------------------------------------
In this section we discuss two recent types of regularization methods that generalize the classical Tikhonov method and that were first introduced and studied in \cite{hochstenbach2011fractional} and \cite{klann2008regularization}.

%-------------------------------------------------------------------
\subsection{Weighted Tikhonov regularization}\label{sec:weighted} %label n.7
%-------------------------------------------------------------------

\begin{definition}[\cite{hochstenbach2011fractional}]\label{def:weightedTik}
We call \emph{Weighted Tikhonov} method the filter based method
\[
R_{\alpha, r} y := \sum_{m:\, \sigma_m >0} F_{\alpha, r} (\sigma_m) \sigma_m ^{-1} \langle y, u_m \rangle v_m,
\]
where the filter function is
\begin{equation}
F_{\alpha, r} (\sigma) = \frac{\sigma ^{r + 1}}{\sigma ^{r +1} + \alpha}, \label{eq:3.2}
\end{equation}
for $\alpha >0$ and $r \geq 0$.
\end{definition}

According to \eqref{eq:xalpha} and \eqref{eq:xdeltaalpha}, we fix the following notation
\begin{align}
x_{\alpha, r} & := R_{\alpha, r} y, \qquad y \in \mathcal{D}(K^\dagger), \label{eq:xweight}\\
x_{\alpha, r}^\delta & := R_{\alpha, r} y^\delta, \qquad y^\delta \in \Y.  \label{eq:xdweight}
\end{align}

\begin{remark}\label{rem:weightedTik}
The Weighted Tikhonov method can also be defined as the unique minimizer of the following functional,
\begin{equation}\label{eq:7.1}
R_{\alpha, r} y := {\rm argmin}_{x \in X} \left\{ \| Kx -y \|_W + \alpha \|x\|   \right\},
\end{equation}
where the semi-norm $\| \cdot \|_W$ is induced by the operator $W:= \left( K K^* \right)^{\frac{r-1}{2}}$. For $0 \leq r < 1$, $W$ is to be intended as the Moore-Penrose (pseudo) inverse. Developing the calculations, it follows that
\begin{equation}\label{eq:7.2}
R_{\alpha, r} y = \left[ \left( K^* K \right)^{\frac{r+1}{2}}  + \alpha I \right]^{-1} \left( K^* K \right)^{\frac{r-1}{2}}K^* y.
\end{equation}
That is the reason that motivated us to rename the original method of Hochstenbach and Reichel, that appeared in \cite{hochstenbach2011fractional}, into \emph{weighted Tikhonov method}. In this way it would be easier to distinguish from the \emph{fractional Tikhonov method} introduced by Klann and Ramlau in \cite{klann2008regularization}.
\end{remark}

The optimal order of the weighted Tikhonov regularization was proved in \cite{Gert2015fractional}. The following proposition restates such result, putting in evidence the dependence on $r$ of $\nu$, 
%to obtain the best possible convergence rate and at the same time 
and provides a converse result.

\begin{proposition}%[Order optimality and converse result for weighted Tikhonov filter]
Let $K$ be a compact linear operator with infinite dimensional range. For every given $r\geq 0$ the weighted Tikhonov method, $R_{\alpha, r}$, is a regularization method
of optimal order, under the a-priori assumption $x^\dagger \in X_{\nu, \rho}$, with $0< \nu \leq r +1$. The best possible rate of convergence with respect to $\delta$ is $\|x^\dagger - x_{\alpha, r}^\delta \| = O \left( \delta^{\frac{r +1}{r +2}}\right)$, that is obtained for $\alpha = \left( \frac{\delta}{\rho} \right)^{\frac{r+1}{\nu +1}}$ with $\nu = r +1$. On the other hand, if $\| x^\dagger - x_{\alpha, r} \| = O(\alpha)$ then $x^\dagger \in \X_{r+1}$.
\end{proposition}
\begin{proof}
For weighted Tikhonov the left-hand side of condition \eqref{eq:2.1} becomes
$$
\sup_{0< \sigma \leq \sigma_1} \left| \frac{\sigma^{r}}{\sigma^{r +1} + \alpha} \right|.
$$
By derivation,  if $r >0 $ then it is straightforward to see that the quantity above is bounded
by $\alpha ^{-\beta}$, with $\beta = 1/(r +1)$.
Similarly, the left-hand side of condition \eqref{eq:2.2} takes the
form
$$
\sup_{0 \leq \sigma \leq \sigma_1} \left|   \frac{\alpha \sigma^\nu}{\sigma^{r +1} + \alpha } \right|,
$$
and it is easy to check that it is bounded by $\alpha ^{\beta
\nu}$ if and only if $0 < \nu \leq r +1$. From
Theorem~\ref{proposition:2.1}, as long as $0 < \nu \leq r +1$,
with $r > 0$, if $x^\dagger \in \X_{\nu, \rho}$ then we find order
optimality \eqref{eq:1.6} and the best possible rate of convergence obtainable with respect to $\delta$ is
$O(\delta ^{\frac{r +1}{\nu +1}})$, for $\nu = r +1$.

On the contrary, with $\beta = 1/(r +1)$ and $\nu = r+1$, we deduce that
$$
\left| \left( 1 - F_{\alpha,r}(\sigma) \right)\sigma^{\nu} \right|
= \frac{\alpha \sigma^{\nu}}{\sigma^{r +1} + \alpha} \geq
\frac{1}{2} \alpha, \qquad \mbox{for } \sigma \in [\alpha^\beta,
\sigma_1].
$$
Therefore, if $\| x^\dagger - x_{\alpha,
r}\| = O (\alpha)$ then $x^\dagger
\in \X_{\nu}$ by Theorem \ref{thm:con.result}.
\end{proof}

%--------------------------------------------------------------------------
\subsection{Fractional Tikhonov regularization} \label{sec:fractional}
%--------------------------------------------------------------------------
Here we introduce the \emph{fractional Tikhonov} method defined and discussed in \cite{klann2008regularization}.

\begin{definition}[\cite{klann2008regularization}]\label{def:fract}
We call \emph{Fractional Tikhonov} method the filter based method
\[
R_{\alpha, \gamma} y := \sum_{m:\, \sigma_m >0} F_{\alpha, \gamma} (\sigma_m) \sigma_m ^{-1} \langle y, u_m \rangle v_m,
\]
where the filter function is
\begin{equation}
F_{\alpha, \gamma} (\sigma) = \frac{\sigma ^{2\gamma}}{(\sigma ^2 + \alpha)^\gamma}, \label{eq:3.1b}
\end{equation}
for $\alpha >0$ and $\gamma \geq 1/2$.
\end{definition}

Note that  $F_{\alpha, \gamma} $ is well-defined also for $0<\gamma<1/2$, but the condition \eqref{eq:1.3} requires 
 $\gamma \geq 1/2$ to guarantee that $F_{\alpha, \gamma} $ is a filter function.

We use the notation for $x_{\alpha, \gamma}$ and $x_{\alpha, \gamma}^\delta$ like in equations 
\eqref{eq:xweight} and \eqref{eq:xdweight}, respectively.
The optimal order of the fractional Tikhonov regularization was proved in \cite[Proposition 3.2]{klann2008regularization}. 
The following proposition restates such result including also $\gamma=1/2$
%to obtain the best possible convergence rate and at the same time 
and provides a converse result.

\begin{proposition}
The extended fractional Tikhonov filter method is a regularization method of optimal order, under the a-priori assumption $x^\dagger \in X_{\nu, \rho}$, for every $\gamma \geq 1/2$ and $0 < \nu \leq 2$. The best possible rate of convergence with respect to $\delta$ is $\|x^\dagger - x_{\alpha, \gamma}^\delta \| = O \left( \delta^{\frac{2}{3}}\right)$, that is obtained for $\alpha = \left( \frac{\delta}{\rho} \right)^{\frac{2}{\nu +1}}$ with $\nu = 2$. On the other hand, if $\| x^\dagger - x_{\alpha, \gamma} \| = O(\alpha)$ then $x^\dagger \in \X_2$.
\end{proposition}
\begin{proof}
Condition \eqref{eq:1.3} is verified for $\gamma \geq 1/2$ and the same holds for conditions \eqref{eq:1.4} and \eqref{eq:1.5}. Deriving the filter function, it is immediate to see that equation \eqref{eq:2.1} is verified for $\gamma \geq 1/2$, with $\beta = 1/2$. It remains to check equation \eqref{eq:2.2}: 
\begin{align*}
\left(1- F_{\alpha, \gamma}(\sigma)  \right) \sigma^\nu &= \frac{\left( \sigma^2 + \alpha \right)^\gamma - \sigma^{2\gamma}}{\left(  \sigma^2 + \alpha \right)^\gamma}\sigma^\nu \\
&=  \frac{\left( \frac{\sigma^2}{\alpha} + 1 \right)^\gamma - \left( \frac{\sigma^2}{\alpha} \right)^\gamma}{\left(  \frac{\sigma^2}{\alpha} + 1 \right)^{\gamma-1}} \cdot \frac{\alpha \sigma^\nu}{\sigma^2 + \alpha} \\
&=  h\left(\frac{\sigma^2}{\alpha}\right) \cdot \left( 1- F_{\alpha, 1}(\sigma) \right) \sigma^\nu,
\end{align*}
where $h(x)=\frac{(x+1)^\gamma-x^\gamma}{(x+1)^{\gamma-1}}$ is monotone, $h(0)=1$ for every $\gamma$, and $\lim_{x\to \infty} h(x) = \gamma$. 
Namely $h(x) \in (\gamma, 1]$ for $0 \leq \gamma \leq 1$ and $h(x) \in [1,\gamma)$ for $\gamma \geq 1$. Therefore we deduce that
\begin{align}\label{eq:4.2a}
& \gamma \left(  1 - F_{\alpha, 1}(\sigma)\right) \leq \left( 1 - F_{\alpha, \gamma}(\sigma) \right) \leq \left( 1 - F_{\alpha, 1}(\sigma) \right), \qquad \mbox{for } 0\leq \gamma \leq 1, \\
& \left(  1 - F_{\alpha, 1}(\sigma)\right) \leq \left( 1 - F_{\alpha, \gamma}(\sigma) \right) \leq \gamma \left( 1 - F_{\alpha, 1}(\sigma) \right), \qquad \mbox{for } \gamma \geq 1, \label{eq:4.2b}
\end{align}
from which we infer that
\begin{equation}
\sup_{\sigma \in [0, \sigma_1]} \left| \left(1- F_{\alpha, \gamma}(\sigma)  \right) \sigma^\nu \right| \leq \max\{1,\gamma\}  \sup_{\sigma \in [0, \sigma_1]} \left| \left(1- F_{\alpha, 1}(\sigma)  \right) \sigma^\nu \right| \leq c \alpha^{\frac{\nu}{2}},
\end{equation}
since $F_{\alpha, 1} (\sigma)$ is standard Tikhonov, that is of optimal order, with $\beta= 1/2$ and for every $0 < \nu \leq 2$, see \cite{engl1996regularization}.
On the contrary, with $\beta = 1/2$ and $\nu = 2$, and by equations \eqref{eq:4.2a} and \eqref{eq:4.2b}, we deduce that
\begin{equation}
\left(1 - F_{\alpha, \gamma} (\sigma)  \right) \sigma^2 \geq \min \{1, \gamma  \} \left(1 - F_{\alpha, 1} (\sigma)  \right) \sigma^2 \geq \frac{1}{2}\alpha, \qquad \mbox{for } \sigma \in [\alpha^{\frac{1}{2}}, \sigma_1].
\end{equation}
Therefore, if $\| x^\dagger - x_{\alpha,
r}\| = O (\alpha)$ then $x^\dagger
\in \X_{2}$ by Theorem \ref{thm:con.result}.
\end{proof}

%--------------------------------------------------------------------------
\section{Saturation results}\label{sec:saturation}
%--------------------------------------------------------------------------

The following proposition deals with a saturation result similar to a well known result for classic Tikhonov, cf. \cite[Proposition 5.3]{engl1996regularization}.

\begin{proposition}[Saturation for weighted Tikhonov regularization]\label{prop:satw}
Let $K : \X \to \Y$ be a compact linear operator with infinite dimensional range and $R_{\alpha, r}$ be the corresponding family of weighted Tikhonov regularization operators in Definition~\ref{def:weightedTik}. Let $\alpha = \alpha(\delta, y^\delta)$ be any parameter choice rule. If
\begin{equation}\label{eq:3.8}
\sup \left\{ \| x_{\alpha, r}^\delta - x^\dagger \| :\, \|Q( y - y^\delta) \| \leq \delta  \right\} = o(\delta^{\frac{r+1}{r+2}}),
\end{equation}
then $x^\dagger =0$, where we indicated with $Q$ the orthogonal projector onto $\overline{R(K)}$.
\end{proposition}
\begin{proof}
Define 
\begin{align*}
\delta_m &:= \sigma_m ^{r+2}, 
&y_m^\delta &:= y + \delta_m u_m \mbox{ so that } \|y - y_m^\delta \| \leq \delta_m, \\
\alpha_m &:= \alpha (\delta_m, y_m^\delta), 
&x_{m}&:=x_{\alpha_m , r},\hspace{3cm}
x_{m}^\delta := x_{\alpha_m , r} ^{\delta_m}.
\end{align*}
By the assumption that $K$ has not finite dimensional range, then $\sigma_m >0$ for every $m$ and $\lim_{m \to \infty} \sigma_m =0$. %In \cite{hochstenbach2011fractional} it is shown that weighted Tikhonov filter comes from minimizing the functional $x \mapsto \| Kx -y \|_{W} + \alpha \|x\|$, where the semi-norm $\| \cdot \|_W$ is induced by the semi-definite positive operator $(KK^*)^{\frac{r-1}{2}}$ (where for $ 0 <r < 1$ it has to be meant as the Moore-Penrose generalized inverse). Minimizing that functional brings to the normal equation
%\begin{equation}\label{eq:3.5}
%\left( (K^* K)^{\frac{r+1}{2}} + \alpha I \right) x_{\alpha, r} = (K^* K)^{\frac{r-1}{2}}K^* y,
%\end{equation}
%namely
%$$
%x_{\alpha, r} = R_{\alpha,r}  y = \left( (K^* K)^{\frac{r+1}{2}} + \alpha I \right)^{-1} (K^* K)^{\frac{r-1}{2}}K^* y,
%$$
According to Remark~\ref{rem:weightedTik},
%from which we obtain the filter functions \eqref{eq:3.2} that define the method. 
from equation \eqref{eq:7.2} we have 
$$
x_m^\delta - x^\dagger = R_{\alpha_m, r} y_m^\delta - x^\dagger =
R_{\alpha_m, r} y + \delta_mR_{\alpha_m, r}u_m - x^\dagger =  x_m - x^\dagger + \delta_mF_{\alpha_m, r}(\sigma_m)\sigma_m^{-1}v_m
%= x_m - x^\dagger + \delta_m \sigma_m \left( (K^* K)^{\frac{r+1}{2}} + \alpha I \right)^{-1} (K^* K)^{\frac{r-1}{2}} v_m,
$$
and hence by \eqref{eq:3.2}
$$
\| x_m ^\delta - x^\dagger \|^2 = \| x_m - x^\dagger \|^2 + 2 \frac{\delta_m \sigma_m ^r}{\sigma_m ^{r+1}+\alpha_m}\langle x_m - x^\dagger, v_m \rangle + \left( \frac{\delta_m \sigma_m ^r}{\sigma_m ^{r+1}+\alpha_m} \right)^2.
$$
From the choice of $\delta_m:=\sigma_m^{r+2}$ follows that
\begin{align}\label{eq:3.4.1}
\left( \delta_m ^{-\frac{r+1}{r+2}} \| x_m ^\delta - x^\dagger \| \right)^2 &\geq  \frac{2}{\delta_m ^{\frac{r+1}{r+2}}+\alpha_m}\langle x_m - x^\dagger, v_m \rangle + \left( \frac{\delta_m ^{\frac{r+1}{r+2}}}{\delta_m ^{\frac{r+1}{r+2}}+\alpha_m} \right)^2 \notag \\
&=   \frac{2}{1+ \delta_m ^{-{\frac{r+1}{r+2}}}\alpha_m}\delta_m ^{-{\frac{r+1}{r+2}}} \langle x_m - x^\dagger, v_m \rangle + \left( \frac{1}{1 + \delta_m ^{-{\frac{r+1}{r+2}}}\alpha_m} \right)^2.
\end{align}
By \eqref{eq:7.2},
\begin{align}
\left( (K^* K)^{\frac{r+1}{2}} + \alpha_m I \right) (x^\dagger - x_m ^\delta) &= 
\left(K^* K  \right)^{\frac{r+1}{2}} x^\dagger + \alpha_m x^\dagger - \left( K^* K \right)^{\frac{r-1}{2}}K^* y_m^\delta \notag \\
&= \alpha_m x^\dagger - \delta_m (K^* K)^{\frac{r-1}{2}}K^* u_m,
\end{align}
so that
\begin{equation}\label{eq:3.6}
\alpha_m \| x^\dagger \| = O (\delta_m + \| x^\dagger - x_m ^\delta \| ).
\end{equation}
Since, by assumption, $ \| x^\dagger - x_m ^\delta \| = o (\delta_m ^{\frac{r+1}{r+2}})$, it follows from \eqref{eq:3.6} that if $x^\dagger \neq 0$, then
\begin{equation}\label{eq:3.7}
\lim_{m \to \infty} \alpha_m \delta_m ^{-\frac{r+1}{r+2}}=0.
\end{equation}
Now, by \eqref{eq:3.8} and \eqref{eq:3.7} applied to inequality \eqref{eq:3.4.1} it follows that $0 \geq 1$,which is a contradiction. Hence $x^\dagger =0$.
\end{proof}

Note that for $r=1$ (classical Tikhonov) the previous proposition gives exactly Proposition~5.3 in \cite{engl1996regularization}. On the other hand, taking a large $r$, it is possible to overcome the saturation result of classical Tikhonov obtaining a convergence rate arbitrary close to $O(\delta)$. 

A similar saturation result can be proved also for the fractional Tikhonov regularization in Definition~\ref{def:fract}. 

\begin{proposition}[Saturation for fractional Tikhonov regularization]\label{prop:satfract}
Let $K : \X \to \Y$ be a compact linear operator with infinite dimensional range and let $R_{\alpha, \gamma}$ be the corresponding family of fractional Tikhonov regularization operators in Definition~\ref{def:fract}, with fixed $\gamma \geq 1/2$. Let $\alpha = \alpha(\delta, y^\delta)$ be any parameter choice rule. If
\begin{equation}\label{eq:3.9}
\sup \left\{ \| x_{\alpha, \gamma}^\delta - x^\dagger \| :\, \|Q( y - y^\delta) \| \leq \delta  \right\} = o(\delta^{\frac{2}{3}}),
\end{equation}
then $x^\dagger =0$, where we indicated with $Q$ the orthogonal projector onto $\overline{R(K)}$.
\end{proposition}
\begin{proof}
If $\gamma =1$, the thesis follows from the saturation result for standard Tikhonov \cite[Proposition 5.3]{engl1996regularization}.
For $\gamma \neq 1$, recalling that 
$$
x_{\alpha, \gamma} - x^\dagger  = \sum_{\sigma_m >0} \left( F_{\alpha,\gamma}(\sigma_m) -1 \right)\sigma_m^ {-1} \langle y, u_m \rangle v_m,
$$
by equations \eqref{eq:4.2a} and \eqref{eq:4.2b}, we obtain
\begin{equation}
\| x_{\alpha, \gamma} - x^\dagger \| > c \| x_{\alpha, 1} - x^\dagger \|,
\end{equation}
where $c = \min \{1, \gamma  \}$ and $x_{\alpha, 1}$ is standard Tikhonov. Let us define 
\[
\phi_\gamma (y) := \| x_{\alpha, \gamma} - x^\dagger \|.
\]
%\begin{align*}
%&\phi_\gamma (y) := \| x_{\alpha, \gamma} - x^\dagger \|\\
%&\phi_1 (y) := \| x_{\alpha, 1} - x^\dagger \|.
%\end{align*}
Then, by the continuity of %$\phi_1$ and
 $\phi_\gamma$, there exists $\delta>0$ such that, for every $y^\delta \in \overline{B}_\delta (y)$,  we find
$$
\phi_\gamma (y^\delta) > c \cdot \phi_1 (y^\delta),
$$
with $ \overline{B}_\delta (y)$ being the closure of the ball of center $y$ and radius $\delta$.
Passing to the $\sup$ we obtain that
%$$
%\sup_{y^\delta \in \overline{B}_\delta (y)}  \phi_\gamma (y^\delta) \geq c \cdot\sup_{y^\delta \in \overline{B}_\delta (y)} \phi_1 (y^\delta),
%$$
%namely,
\begin{equation}
\sup \left\{ \| x_{\alpha, \gamma}^\delta - x^\dagger \| :\, \|Q( y - y^\delta) \| \leq \delta  \right\} \geq c \cdot \sup \left\{ \| x_{\alpha, 1}^\delta - x^\dagger \| :\, \|Q( y - y^\delta) \| \leq \delta  \right\}.
\end{equation}
Therefore, using relation \eqref{eq:3.9}, we deduce 
\begin{equation}
\sup \left\{ \| x_{\alpha, 1}^\delta - x^\dagger \| :\, \| y - y^\delta \| \leq \delta  \right\} = o(\delta^{\frac{2}{3}}),
\end{equation}
and the thesis follows again from the saturation result for standard Tikhonov, cf. \cite[Proposition~5.3]{engl1996regularization}.
\end{proof}

Differently from the weighted Tikhonov regularization, for the fractional Tikhonov method, it is not possible to overcome the saturation result of classical Tikhonov, even for a large $\gamma$.

%--------------------------------------------------------------------------
\section{Stationary iterated regularization}\label{sec:statiter}
%--------------------------------------------------------------------------
We define new iterated regularization methods based on weighed and fractional Tikhonov regularization using 
the same iterative refinement strategy of iterated Tikhonov regularization \cite{brill1987iterative,engl1996regularization}.
We will show that the iterated methods go beyond  the saturation results proved in the previous section.
In this section the regularization parameter will still be $\alpha$ with the iteration step, $n$, assumed to be fixed. On the contrary, in Section~\ref{sec:NSIWT}, we will analyze the nonstationary counterpart of this iterative method, in which $\alpha$ will be replaced by a pre-fixed sequence $\{ \alpha_n \}$ and we will be concerned on the rate of convergence with respect to the index $n$.  

%--------------------------------------------------------------------------
\subsection{Iterated weighted Tikhonov regularization}
%--------------------------------------------------------------------------
We propose now an iterated regularization method based on weighted Tikhonov%, in the same spirit of iterated Tikhonov, 

\begin{definition}[Stationary iterated weighted Tikhonov]
We define the \emph{stationary iterated weighted Tikhonov method} (\emph{SIWT}) as
\begin{equation}\label{eq:4.1}
\begin{cases}
x_{\alpha, r}^0 := 0; \\
\left( (K^* K)^{\frac{r+1}{2}} + \alpha I  \right)x_{\alpha, r}^n := (K^* K)^{\frac{r-1}{2}} K^* y + \alpha x_{\alpha, r}^{n-1},
\end{cases}
\end{equation}
with $\alpha >0$ and $r \geq 0$, or equivalently 
\begin{equation}
\begin{cases}
x_{\alpha, r}^0 := 0 \\
x_{\alpha, r}^n := \mbox{ \emph{argmin}}_{x \in X} \left\{ \| Kx - y \|_W + \alpha \|x - x_{\alpha, r}^{n-1} \|  \right\},
\end{cases}
\end{equation}
where $\| \cdot \|_W$ is the semi-norm introduced in \eqref{eq:7.1}.
We define $x_{\alpha, r}^{n, \delta}$ as the $n$-th iteration of weighted Tikhonov if $y = y^\delta$. 
\end{definition}

\begin{proposition}
For any given $n \in \mathbb{N}$ and $r>0$, the SIWT in \eqref{eq:4.1} is a filter based regularization method, with filter function
\begin{equation}
F_{\alpha, r}^{(n)} (\sigma) = \frac{(\sigma^{r+1} + \alpha)^n - \alpha^n}{(\sigma^{r+1} + \alpha)^n}.
\end{equation}
Moreover, the method is of optimal order, under the a-priori assumption $x^\dagger \in X_{\nu, \rho}$, for $r >0$ and $0 < \nu \leq n(r+1)$, with best convergence rate $\| x^\dagger - x_{\alpha, r}^{n, \delta} \| = O (\delta ^{\frac{n(r+1)}{1+n(r+1)}})$, 
that is obtained for $\alpha =  (\frac{\delta}{\rho})^{\frac{n(r+1)}{1+\nu}}$, with $\nu=n(r+1)$. 
On the other hand, if $\|x^\dagger - x_{\alpha, r}^{n} \| = O(\alpha^n)$, then $x^\dagger \in X_{n(r+1)}$.
\end{proposition}
\begin{proof}
Multiplying both sides of \eqref{eq:4.1} by $\left( (K^* K)^{\frac{r+1}{2}} + \alpha I  \right)^{n-1}$ and iterating the process, we get
\begin{align*}
\left( (K^* K)^{\frac{r+1}{2}} + \alpha I  \right)^n x_{\alpha, r}^n &= \left\{ \sum_{j=0}^{n-1} \alpha^j \left((K^*K)^{\frac{r+1}{2}} +\alpha I \right)^{n-1-j}  \right\} (K^*K)^{\frac{r-1}{2}}K^* y\\
&= \left[  \left((K^*K)^{\frac{r+1}{2}} + \alpha I \right)^n - \alpha^n I  \right](K^*K)^{-1} K^* y.
\end{align*}
%Now, by Remark \ref{remark:1.1}, it follows that
Therefore, the filter function in \eqref{eq:def_filter} is equal to
$$
F_{\alpha, r}^{(n)} (\sigma) = \frac{(\sigma^{r+1} + \alpha)^n - \alpha^n}{(\sigma^{r+1} + \alpha)^n},
$$
as we stated. Condition \eqref{eq:1.5} is straightforward to verify. Moreover, note that
\begin{align*}
 F_{\alpha, r}^{(n)}  (\sigma)  &=   \frac{(\sigma^{r+1} + \alpha)^n - \alpha^n}{(\sigma^{r+1} + \alpha)^n}  \\
&=  \frac{\sigma^{r+1}}{\sigma^{r+1}+\alpha}  \cdot  \frac{ \left( \sum_{j=0}^{n-1} \alpha^{j}(\sigma^{r+1} + \alpha)^{n-1-j}  \right)}{(\sigma^{r+1} + \alpha)^{n-1}}      \\
&=  F_{\alpha, r}(\sigma) \cdot \left( 1 + \left( \frac{\alpha}{\sigma^{r+1}+\alpha} \right) + \cdots + \left( \frac{\alpha}{\sigma^{r+1}+\alpha}  \right)^{n-1} \right),
\end{align*}
from which it follows that
\begin{equation}
F_{\alpha, r}(\sigma) \leq F_{\alpha, r}^{(n)}  (\sigma) \leq n F_{\alpha, r}(\sigma).
\end{equation}
Therefore, conditions \eqref{eq:1.3}, \eqref{eq:1.4} and \eqref{eq:2.1} follows immediately by the regularity of the weighted Tikhonov filter method for $r>0$ and by the order optimality for $r >0$. Finally, condition \eqref{eq:2.2} becomes
$$
\sup_{\sigma \in [0,\sigma_1]} \left| \frac{\alpha^n \sigma^\nu}{(\sigma^{r+1}+\alpha)^n}  \right|,
$$
and deriving one checks that it is bounded by $\alpha ^{\beta \nu}$, with $\beta = 1/(r+1)$, if and only if $0 < \nu \leq n(r+1)$. Applying now Proposition \ref{proposition:2.1} the rest of the thesis follows.

On the contrary, if we define $\beta = 1/(r+1)$ and $\nu = n(r+1)$, then we deduce that
$$
\left( 1 - F_{\alpha, r}^{(n)} (\sigma)  \right)\sigma^\nu  =  \frac{\alpha^n \sigma^\nu}{(\sigma^{r+1}+\alpha)^n} \geq \frac{1}{2^n} \alpha^n \qquad \mbox{for } \sigma \in [\alpha^\beta, \sigma_1].
$$
Therefore, if $\|x^\dagger - x_{\alpha, r}^{n} \| = O(\alpha^n)$, then by Theorem \ref{thm:con.result} it follows that $x^\dagger \in X_{n(r+1)}$.
\end{proof}

 If $n$ is large, then we note that the convergence rate approaches $O(\delta)$ also for a fixed small $r$. 
The study of the convergence for increasing  $n$ and fixed $\alpha$ will be dealt with in Section \ref{sec:NSIWT}. 

%--------------------------------------------------------------------------
\subsection{Iterated fractional Tikhonov regularization}  %\label n.5
%--------------------------------------------------------------------------
With the same path as in the previous subsection, we propose here the stationary iterated version of the fractional Tikhonov method.

\begin{definition}[Stationary iterated fractional Tikhonov]
We define the \emph{stationary iterated fractional Tikhonov method} (\emph{SIFT}) as
\begin{equation}\label{eq:4.3}
\begin{cases}
x_{\alpha, \gamma}^0 := 0; \\
\left( K^* K + \alpha I  \right)^\gamma x_{\alpha, \gamma}^n := (K^* K)^{\gamma -1} K^* y + \left[\left( K^*K + \alpha I \right)^\gamma - \left( K^*K \right)^\gamma  \right] x_{\alpha, \gamma}^{n-1},
\end{cases}
\end{equation}
with $\gamma \geq 1/2$.
We define $x_{\alpha, \gamma}^{n, \delta}$ for the $n$-th iteration of fractional Tikhonov if $y = y^\delta$. 
\end{definition}

\begin{proposition}\label{it-frac}
For any given $n \in \mathbb{N}$ and $\gamma \geq 1/2$, the SIFT in \eqref{eq:4.3} is a filter based regularization method, with filter function
\begin{equation}\label{eq:5.1}
F_{\alpha, \gamma}^{(n)}  (\sigma) = \frac{\left(\sigma^{2} + \alpha \right)^{\gamma n} - \left[ \left( \sigma^2 + \alpha \right)^\gamma - \sigma^{2\gamma} \right]^n}{\left(\sigma^{2} + \alpha\right)^{\gamma n}}.
\end{equation}
Moreover, the method is of optimal order, under the a-priori assumption $x^\dagger \in X_{\nu, \rho}$, for $\gamma \geq 1/2$ and $0 < \nu \leq 2n$, with best convergence rate $\| x^\dagger - x_{\alpha, \gamma}^{n, \delta} \| = O (\delta ^{\frac{2n}{2n +1}})$, 
that is obtained for $\alpha =  (\frac{\delta}{\rho})^{\frac{2n}{\nu+1}}$, with $\nu=2n$. 
On the other hand, if $\|x^\dagger - x_{\alpha, \gamma}^{n} \| = O(\alpha^n)$, then $x^\dagger \in X_{2n}$.
\end{proposition}
\begin{proof}
Multiplying both sides of \eqref{eq:5.1} by $\left( K^* K + \alpha I  \right)^{(n-1)\gamma}$ and iterating the process, we get
\begin{align*}
\left( K^* K + \alpha I  \right)^{n\gamma} x_{\alpha, \gamma}^n &= \left\{ \sum_{j=0}^{n-1}  \left(K^*K +\alpha I \right)^{j\gamma} \left[ \left( K^*K + \alpha I \right)^\gamma - \left(K^*K \right)^\gamma \right]^{n-1 - j} \right\} (K^*K)^{\gamma -1}K^* y\\
&= \left\{ \left(K^*K + \alpha I \right)^{\gamma n} - \left[\left( K^*K + \alpha I \right)^\gamma - \left(K^*K \right)^\gamma  \right]^n \right\} (K^*K)^{-1} K^* y,
\end{align*}
where we used the fact that $ \left(K^*K +\alpha I \right)^{-\gamma}$ and $\left[ \left( K^*K + \alpha I \right)^\gamma - \left(K^*K \right)^\gamma \right]$ commute. 
%Now, by Remark \ref{remark:1.1}, it follows that
Therefore, the filter function in \eqref{eq:def_filter} is given by
$$
F_{\alpha, \gamma}^n (\sigma) = \frac{(\sigma^{2} + \alpha)^{\gamma n} - \left[ \left(\sigma^2 + \alpha  \right)^\gamma - \sigma^{2\gamma} \right]^n}{\left(\sigma^{2} + \alpha \right)^{\gamma n}},
$$
as we stated.
We observe that
\begin{align}
F_{\alpha, \gamma}^{(n)}  (\sigma) &=  \frac{(\sigma^{2} + \alpha)^{\gamma n} - \left[ \left(\sigma^2 + \alpha  \right)^\gamma - \sigma^{2\gamma} \right]^n}{\left(\sigma^{2} + \alpha \right)^{\gamma n}} \nonumber \\
							  &= \frac{\sigma^{2\gamma}}{(\sigma^2 + \alpha)^\gamma} \cdot \frac{1}{(\sigma^2 + \alpha)^{\gamma(n-1)}} \cdot \sum_{j=0}^{n-1} (\sigma^2 + \alpha)^{\gamma j} \left[ (\sigma^2 + \alpha)^\gamma - \sigma^{2\gamma}\right]^{n-1-j} \nonumber \\
							  &= \frac{\sigma^{2\gamma}}{(\sigma^2 + \alpha)^\gamma} \cdot \left\{ 1 + \left[1 - \left(\frac{\sigma^{2}}{\sigma^2 + \alpha}\right)^\gamma  \right] + \cdots + \left[1 - \left(\frac{\sigma^{2}}{\sigma^2 + \alpha}\right)^\gamma  \right]^{n-1}  \right\}, \nonumber						
\end{align}
from which we deduce that 
\begin{equation} 
 F_{\alpha, \gamma}^{(n)}  (\sigma) \leq n F_{\alpha, \gamma}(\sigma).
\end{equation}
Therefore, since $F_{\alpha,\gamma}$ is a regularization method of optimal order, conditions \eqref{eq:1.3}, \eqref{eq:1.4} and \eqref{eq:2.1} are satisfied. Moreover, it is easy to check condition \eqref{eq:1.5} and so we get the regularity for the method. It remains to check condition \eqref{eq:2.2} for the order optimality.

From equations \eqref{eq:4.2a} and \eqref{eq:4.2b} we deduce that
\begin{align}
1 - F_{\alpha, \gamma}^{(n)}  (\sigma) &= \left[\frac{(\sigma^2 + \alpha)^\gamma - \sigma^{2\gamma}}{(\sigma^2 + \alpha)^\gamma} \right]^n \nonumber \\
&= \left[ 1 - \frac{\sigma^{2\gamma}}{(\sigma^2 + \alpha)^\gamma} \right]^n \nonumber \\
&= \left( 1 - F_{\alpha, \gamma}(\sigma) \right)^n    \label{eq:5.3}\\
&\leq \left(\max \{1, \gamma\} \right)^n  \left( 1- F_{\alpha, 1}(\sigma) \right)^n \nonumber\\
&= c \left(1 - F_{\alpha, 1}^n(\sigma)  \right), \nonumber
\end{align}
where $F_{\alpha,1}(\sigma)$ is the standard Tikhonov filter and $F_{\alpha,1}^{(n)}  (\sigma)$ is the filter function of the stationary iterated Tikhonov, i.e., $F_{\alpha, 1}^{(n)}  (\sigma) = \frac{(\sigma^2 + \alpha)^n - \alpha^n}{(\sigma^2 + \alpha)^n}$. Now condition \eqref{eq:2.2} follows from the properties of stationary iterated Tikhonov, with $\beta = 1/2$ and $0 < \nu \leq 2n$, see \cite[p. 124]{hankehansen1993}.  By applying  Proposition~\ref{proposition:2.1} we get the best convergence rate, $O (\delta ^{\frac{2n}{2n +1}})$.

On the contrary, set $\beta = 1/2$ and $\nu = 2n$. First, let us observe that from equations \eqref{eq:5.3} and \eqref{eq:4.2a}, \eqref{eq:4.2b}, we infer that
$$
1 - F_{\alpha, \gamma}^{(n)} (\sigma) \geq \left(\min\{ 1, \gamma \}\right)^n \left(1 - F_{\alpha, 1}^{(n)}  (\sigma)\right).
$$
Then, we deduce that
\begin{align*}
\left(1 - F_{\alpha, \gamma}^{(n)}  (\sigma)  \right) \sigma^\nu &\geq c \frac{\alpha^n \sigma^{2n}}{(\sigma^2 + \alpha)^n}\\
&\geq c \alpha^n \qquad \mbox{for } \sigma \in [\alpha^\beta, \sigma_1].
\end{align*}
Therefore, if $\| x^\dagger - x_{\alpha, \gamma}^n \| = O(\alpha^n)$, then $x^\dagger \in X_{2n}$ by Theorem \ref{thm:con.result}.
\end{proof}

The previous proposition shows that, similarly to SIWT, a large $n$ allows to overcome the saturation result in Proposition~\ref{prop:satfract}.
The study of the convergence for increasing  $n$ and fixed $\alpha$ will be dealt with in Section \ref{sec:NSIFT}. 

%--------------------------------------------------------------------------
\section{Nonstationary iterated weighted Tikhonov regularization}\label{sec:NSIWT} %label n.6
We introduce a nonstationary version of the iteration~\eqref{eq:4.1}.
We study the convergence and we prove that the new iteration is a regularization method.  

%--------------------------------------------------------------------------
\begin{definition}
Let $\{ \alpha_n \}_{n\in \N}, \{ r_n \}_{n \in \N} \subset \R_{>0}$ be sequences of positive real numbers. We define a nonstationary iterated weighted Tikhonov method (NSIWT) as follows
\begin{equation}\label{eq:6.1}
\begin{cases}
x_{\alpha_0, r_0}^0 := 0, \\
\left[ \left( K^* K \right)^{\frac{r_n +1}{2}} + \alpha_n I\right] x_{\alpha_n, r_n}^n := \left( K^* K \right)^{\frac{r_n -1}{2}} K^* y + \alpha_n x_{\alpha_{n-1}, r_{n-1}}^{n-1},
\end{cases}
\end{equation}
or equivalently 
\begin{equation}\label{eq:6.2}
\begin{cases}
x_{\alpha_0, r_0}^0 := 0, \\
x_{\alpha_n, r_n}^n := \mbox{ \emph{argmin}}_{x \in X} \left\{ \| Kx - y \|_{W_{n}} + \alpha_n \|x - x_{\alpha_{n-1}, r_{n-1}}^{n-1} \|  \right\},
\end{cases}
\end{equation}
where $\| \cdot \|_{W_{n}}$ is the semi-norm introduced by the operator $W_n:= \left( K K^* \right)^{\frac{r_n-1}{2}}$ and depending on $n$, due to the non stationary character of $r_n$. %We define $x_{\alpha_n, r_n}^{n, \delta}$ for the $n$-th iteration of weighted Tikhonov if $y = y^\delta$. 
\end{definition}

\subsection{Convergence analysis}
We are concerned about the properties of the sequence $\{ \alpha_n\}$ such that the iteration \eqref{eq:6.1} shall converge. 
To this aim we need some preliminary lemmas, whose proof can be found in the appendix. 

\begin{remark}
Hereafter, without loss of generality, we will consider $\sigma_1 = 1$, namely $\| K\| =1$.
\end{remark}

\begin{lemma}\label{lem:6.1}
Let $\{t_n \}_{n\in \N}$ be a sequence of real numbers such that $0 \leq t_n <1$ for every $n$. Then
\begin{equation}
\prod_{n=1}^\infty (1- t_n) >0 \qquad \mbox{if and only if} \qquad \sum_{n=1}^{\infty} t_n < \infty.
\end{equation}
\end{lemma}
\begin{proof}
See \cite[Theorem 15.5]{rudin1987real}
\end{proof}

\begin{lemma}\label{lem:6.3}
Let $\{ t_k \}_{k \in \N}$ be a sequence of positive real numbers and let $N>0$. Then 
$$
\sum_{k=1}^n t_k \sim c \sum_{k=N}^n t_k,
$$
with $c>0$ (in particular, $c=1$ when $\sum_{k=N}^\infty t_k = \sum_{k=1}^\infty t_k =\infty$).
\end{lemma}

\begin{lemma}\label{lem:6.2}
For every $\lambda \in (0, \infty)$ and for every sequence $\{ t_k \}_{k \in \N} \subset (0, \infty)$ such that $\lim_{k\to \infty}t_k =  t \in (0, \infty]$, we find
$$
\sum_{k=1}^n \frac{1}{t_k} \, \sim \, c_\lambda  \sum_{k=1}^n \frac{\lambda}{\lambda + t_k}, \qquad c_\lambda >0, 
$$
where $\sim$ denotes the asymptotic equivalence.
\end{lemma}

We can now prove a necessary and sufficient condition on the sequence $\{\alpha_n\}$ to have the convergence of NSIWT.

\begin{theorem}\label{thm:6.1}
The \emph{NSIWT} method \eqref{eq:6.1} converges to $x^\dagger\in\X$ 
as $n \to \infty$ if and only if $\sum_{k=1}^n \frac{\sigma^{r_k +1}}{\sigma^{r_k +1} + \alpha_k}$ diverges for every $\sigma >0$.
\end{theorem}
\begin{proof}
Rewriting equation \eqref{eq:6.1} and reminding that $y= Kx^\dagger$, we have
\begin{align*}
x_{\alpha_n, r_n}^n% &= \left[ \left( K^* K \right)^{\frac{r_n +1}{2}}  + \alpha_n I \right]^{-1} \left( K^* K \right)^{\frac{r_n -1}{2}}K^*y + \alpha_n  \left[ \left( K^* K \right)^{\frac{r_n +1}{2}}  + \alpha_n I \right]^{-1} x_{\alpha_{n-1}, r_{n-1}}^{n-1}\\
&= \left[ \left( K^* K \right)^{\frac{r_n +1}{2}}  + \alpha_n I \right]^{-1} \left( K^* K \right)^{\frac{r_n +1}{2}}x^\dagger + \alpha_n  \left[ \left( K^* K \right)^{\frac{r_n +1}{2}}  + \alpha_n I \right]^{-1} x_{\alpha_{n-1}, r_{n-1}}^{n-1}\\
&= \left\{ I - \alpha_n \left[ \left( K^* K \right)^{\frac{r_n +1}{2}}  + \alpha_n I \right]^{-1}  \right\} x^\dagger + \alpha_n \left[ \left( K^* K \right)^{\frac{r_n +1}{2}}  + \alpha_n I \right]^{-1}x_{\alpha_{n-1}, r_{n-1}}^{n-1},
\end{align*}
from which it follows that
\begin{align}
x^\dagger - x_{\alpha_n, r_n}^n &= \alpha_n \left[ \left( K^* K \right)^{\frac{r_n +1}{2}}  + \alpha_n I \right]^{-1} (x^\dagger - x_{\alpha_{n-1}, r_{n-1}}^{n-1}) \nonumber\\
&= (\cdots ) \mbox{ iterating the process $n-1$ times} \nonumber\\
&= \prod_{k=1}^n \alpha_k \left[ \left( K^* K \right)^{\frac{r_k +1}{2}}  + \alpha_k I \right]^{-1} x^\dagger \label{eq:6.3}
\end{align}
since $x_{\alpha_0, r_0}^0:=0$.
As a consequence, the method shall converge if and only if
\begin{equation}
\lim_{n\to \infty} \left \| \prod_{k=1}^n \alpha_k \left[ \left( K^* K \right)^{\frac{r_k +1}{2}}  + \alpha_k I \right]^{-1} x^\dagger \right \| =0
\end{equation}
for every $x^\dagger \in X$, namely, if and only if 
\begin{equation}
\lim_{n\to \infty} \int_{\sigma(K^*K)} \left| \prod_{k=1}^n \frac{\alpha_k}{\sigma^{r_k +1} + \alpha_k} \right|^2 \, d\langle E_{\sigma^2} x^\dagger, x^\dagger \rangle =0
\end{equation}
for every Borel-measure $\langle E x^\dagger, x^\dagger \rangle$ induced by $x^\dagger \in X$. Since 
$$
\left| \prod_{k=1}^n \frac{\alpha_k}{\sigma^{r_k +1}+ \alpha_k} \right|^2 \leq 1
$$
for every $n$, and since
$$
\int_{\sigma(K^*K)} d\langle E_{\sigma^2} x^\dagger, x^\dagger \rangle = \| x^\dagger \|^2,
$$
the Dominated Convergence Theorem \cite[Theorem~1.34, pag. 26]{rudin1987real} implies
\begin{equation}
\lim_{n\to \infty} \int_{\sigma(K^*K)} \left| \prod_{k=1}^n \frac{\alpha_k}{\sigma^{r_k +1}+ \alpha_k} \right|^2 \, d\langle E_{\sigma^2} x^\dagger, x^\dagger \rangle = \int_{\sigma(K^*K)} \lim_{n\to \infty }\left| \prod_{k=1}^n \frac{\alpha_k}{\sigma^{r_k +1}+ \alpha_k} \right|^2 \, d\langle E_{\sigma^2} x^\dagger, x^\dagger \rangle.
\end{equation}
Hence, the NSIWT method is convergent if and only if
\begin{equation}
\prod_{k=1}^\infty \frac{\alpha_k}{\sigma^{r_k +1}+ \alpha_k} = \prod_{k=1}^\infty \left( 1- \frac{\sigma^{r_k +1}}{\sigma^{r_k +1}+ \alpha_k} \right) = 0,
\end{equation}
for $\langle E x^\dagger, x^\dagger \rangle$-a.e. $\sigma^2$, i.e., for every $\sigma \in \sigma(K)\setminus \{ 0 \}$. Applying now Lemma \ref{lem:6.1} the thesis follows.
\end{proof}

\begin{corollary}\label{cor:6.1}
\begin{itemize}
\item[]
\item[(1)] If $\sup_{k \in \N}\{ r_k \} = r \in [0, \infty)$, then the NSIWT method converges if and only if $\sum_{k=1}^n \alpha_k ^{-1}$ diverges.
\item[(2)] Let $\lim_{k \to \infty} r_k = \infty$ monotonically. If $\left(   \sum_{k=1}^n \alpha_k ^{-1} \right)^{-1} = o (\sigma^{r_{n} +1})$ for every $\sigma \in \sigma(K) \setminus \{0\}$, then the NSIWT method converges.
\end{itemize}
\end{corollary}
\begin{proof}
(1) For every $\sigma \in \sigma(K)\setminus \{0\}$, we observe that
\begin{equation}\label{eq:cor25}
\sum_{k=1}^\infty \frac{\sigma^{r+1}}{\sigma^{r +1}+ \alpha_k } \leq \sum_{k=1}^\infty \frac{\sigma^{r_k +1}}{\sigma^{r_k +1}+ \alpha_k } \leq \sum_{k=1}^\infty \frac{1}{1 + \alpha_k } \leq \sum_{k=1}^\infty \frac{1}{ \alpha_k }. 
\end{equation}
If the NSIWT method converges then, by Theorem~\ref{thm:6.1} and by \eqref{eq:cor25}, $\sum_{k=1}^\infty \frac{\sigma^{r_k +1}}{\sigma^{r_k +1}+ \alpha_k }$ diverges and hence $\sum_{k=1}^\infty \frac{1}{\alpha_k } = \infty$. On the other hand, if $\sum_{k=1}^\infty \alpha_k ^{-1} = \infty$, then we can possibly have three different cases: $\lim_{k\to \infty} \alpha_k =0$, $\nexists \lim_{k\to \infty}\alpha_k$ or $\lim_{k\to \infty} \alpha_k \in (0, \infty]$. In the first two cases, $\frac{\sigma^{r+1}}{\sigma^{r +1}+ \alpha_k } \nrightarrow 0$ for every $\sigma >0$, and then the corresponding series diverges. In the latter case instead, by Lemma \ref{lem:6.2}, $\sum_{k=1}^n \frac{1}{\alpha_k} \sim c_{\sigma,r} \sum_{k=1}^n \frac{\sigma^{r+1}}{\sigma^{r+1}+ \alpha_k}$ for every $\sigma >0$. Then, by $\sum_{k=1}^\infty \alpha_k ^{-1} = \infty$, we deduce that $\sum_{k=1}^\infty \frac{\sigma^{r_k +1}}{\sigma^{r_k +1}+ \alpha_k }$ diverges for every $\sigma>0$ and the NSIWT method converges.

(2) We can assume that $0<\sigma <1$. For $\sigma =1$ the result is indeed trivial owing to the equivalence
$$
\sum_{k=1}^\infty \frac{1}{1+\alpha_k} = \infty \Longleftrightarrow  \sum_{k=1}^\infty \alpha_k ^{-1} = \infty \qquad \mbox{(see the previous point)}.
$$
On the other hand, if $\sigma <1$ then we have  $\sigma^{r_n +1} \to 0$ and $\frac{1}{\sigma^{r_n +1} + \alpha_k} \sim \alpha_k ^{-1}$, for $n \to \infty$. Therefore, there exists $N=N(\sigma)$ such that $\frac{1}{\sigma^{r_n +1} + \alpha_k} \geq \frac{1}{2} \alpha_k ^{-1}$ for every $n \geq N$. Hence, we have
$$
\frac{1}{2} \sigma^{r_n +1} \sum_{k=N}^{n} \alpha_k ^{-1} \leq \sigma^{r_n +1} \left(\sum_{k=1}^{N-1} \frac{1}{\sigma^{r_n +1}+ \alpha_k} + \frac{1}{2} \sum_{k=N}^n \alpha_k ^{-1} \right) \leq \sum_{k=1}^n \frac{\sigma^{r_n +1}}{\sigma^{r_n +1} + \alpha_k} \leq \sum_{k=1}^n \frac{\sigma^{r_k +1}}{\sigma^{r_k +1} + \alpha_k}. 
$$
Since, by Lemma \ref{lem:6.3}, $\sum_{k=N}^{n} \alpha_k ^{-1} \sim \sum_{k=1}^{n} \alpha_k ^{-1}$ then, by the preceding inequalities, the hypothesis $\left(\sum_{k=1}^{n} \alpha_k ^{-1}\right)^{-1} = o(\sigma^{r_n +1})$ implies that $\sum_{k=1}^n \frac{\sigma^{r_k +1}}{\sigma^{r_k +1} + \alpha_k} = \infty$ and the NSIWT method converges. 
\end{proof}

Corollary~\ref{cor:6.1} applies immediately to the stationary case, where $\alpha_k = \alpha$ and $r_k =r$ for every $k \in \N$,
showing that SIWT converges.
On the other hand, from point \textit{(2)} of Corollary~\ref{cor:6.1}, given a monotone divergent sequence $r_k \to \infty$ we need a sequence $\alpha_k \to 0$ such that $\alpha_k = o(\sigma^{r_k +1})$ for every $\sigma >0$ in order to preserve the convergence of NSIWT.

Now, we investigate the convergence rate of NSIWT.
\begin{theorem}\label{thm:6.2}
Let $\{x_{\alpha_n, r_n}^n \}_{n \in \N}$ be a convergent sequence of the NSIWT method, with $x^\dagger \in \X_{\nu}$ for some $\nu >0$, and let $\{ \vartheta_n \}_{n \in \N}$ be a divergent sequence of positive real numbers. If
\begin{subequations}
\begin{equation} 
\lim_{n\to \infty} \vartheta_n \sigma^\nu \prod_{k=1}^n \left( 1 - \frac{\sigma^{r_k +1}}{\sigma^{r_k +1} + \alpha_k}  \right)=0 \qquad \mbox{for every } \sigma \in \sigma(K)\setminus \{0\}; \label{eq:6.4}
\end{equation}
\begin{equation}
\sup_{\sigma \in  \sigma(K)\setminus \{0\}} \vartheta_n \sigma^\nu \prod_{k=1}^n \left( 1 - \frac{\sigma^{r_k +1}}{\sigma^{r_k +1} + \alpha_k}  \right) \leq c < \infty \qquad \mbox{ uniformly with respect to } n, \label{eq:6.5}
\end{equation}
\end{subequations}
then 
\begin{equation}
\| x^\dagger - x_{\alpha_n, r_n}^n \| = o (\vartheta_n ^{-1}).
\end{equation}
\end{theorem}
\begin{proof}
From equation \eqref{eq:6.3}, for $x^\dagger \in \X_{\nu}$, we have 
%\[
%\| x^\dagger - x_{\alpha_n, r_n}^n \| =  \left[ \int_{\sigma(K^* K)} \left| \sigma^\nu  \prod_{k=1}^n\left(  \frac{\alpha_k}{\sigma^{r_k +1} + \alpha_k}  \right)  \right|^2 \, d\langle E_{\sigma^2} \omega, \omega \rangle \right]^{1/2}
%\]
%and thus
\begin{align*}
\lim_{n \to \infty} \vartheta_n \| x^\dagger - x_{\alpha_n, r_n}^n \| &= \lim_{n \to \infty}  \left[ \int_{\sigma(K^* K)} \left| \vartheta_n \sigma^\nu  \prod_{k=1}^n\left( 1 - \frac{\sigma^{r_k +1}}{\sigma^{r_k +1} + \alpha_k}  \right)  \right|^2 \, d\langle E_{\sigma^2} \omega, \omega \rangle \right]^{1/2}\\
&= \left[ \int_{\sigma(K^* K)} \left|\lim_{n \to \infty} \vartheta_n \sigma^\nu  \prod_{k=1}^n\left( 1 - \frac{\sigma^{r_k +1}}{\sigma^{r_k +1} + \alpha_k}  \right)  \right|^2 \, d\langle E_{\sigma^2} \omega, \omega \rangle \right]^{1/2},
\end{align*}
by \eqref{eq:6.5} and the Dominated Convergence Theorem. Now, from hypothesis \eqref{eq:6.4}, the thesis follows.
\end{proof}

\begin{corollary}\label{cor:6.2}
We define 
$$
\beta_n = \sum_{k=1}^n \alpha_k ^{-1}, \qquad \tilde{\beta}_n =\sum_{k=1}^n \frac{1}{1+ \alpha_k}.
$$
Let $\{ r_k \}_{k \in \N}$ be a sequence of positive real numbers, $r_k \geq 0$, and let $x^\dagger \in \X_{\nu}$ for some $\nu >0$. If
\begin{itemize}
\item[(i.1)] $\sup_{k \in \N}\{ r_k \} = r \in [0, \infty)$,
\item[(i.2)] $\lim_{n \to \infty} \beta_n = \infty$,
\end{itemize}
then
\begin{subequations}
\label{eq:6.7}
  \begin{align}[left = {\| x^\dagger - x_{\alpha_n, r_n}^n \| = \empheqlbrace}]
    &o (\beta_n ^{- \frac{\nu}{r+1}}) &\mbox{if } \lim_{n\to \infty}\alpha_n = \alpha \in (0, \infty] \label{eq:6.7.1} \\
    &O (\beta_n ^{- \frac{\nu}{r+1}})  &\mbox{if } \lim_{n\to \infty}\alpha_n =0 \mbox{ and } \alpha_n^{-1}\leq c \beta_{n-1}, \, c>0 \label{eq:6.7.2}\\
    &o (\tilde{\beta}_n ^{- \frac{\nu}{r+1}})  &\mbox{otherwise}. \label{eq:6.7.3}
   \end{align}
\end{subequations}
On the contrary, if
\begin{itemize}
\item[(ii.1)] $r_k \to \infty$ monotonically,
\item[(ii.2)] $\beta_n ^{-1}= o(\sigma^{r_n +1})$ for every $\sigma \in \sigma(K)\setminus \{0\}$,
\end{itemize}
then
\begin{equation}\label{eq:6.8}
\| x^\dagger - x_{\alpha_n, r_n}^n \| = o (\beta_n ^{- \frac{\nu}{r_n +1}}).
\end{equation}
\end{corollary}
\begin{proof}
First, note that from (i.1), (i.2) and Corollary \ref{cor:6.1} it follows that the NSIWT method is convergent. Now, since $1-x \leq \textrm{e}^{-x} \leq c_{\nu,r} x^{-\nu/r+1}$, and using (i.2), we have 
\begin{align*}
\sigma^{\nu} \prod_{k=1}^n\left( 1 - \frac{\sigma^{r_k +1}}{\sigma^{r_k +1} + \alpha_k} \right) &\leq \sigma^\nu \textrm{e}^{-\sum_{k=1}^n \frac{\sigma^{r_k +1}}{\sigma^{r_k +1} + \alpha_k}}\\
&\leq  \sigma^\nu \textrm{e}^{-\sigma^{r +1} \sum_{k=1}^n \frac{1}{\sigma^{r +1} + \alpha_k}}\\
&\leq c_{\nu, r} \sigma^\nu  \left( \frac{1}{\sigma^{r +1} \sum_{k=1}^n \frac{1}{\sigma^{r +1} + \alpha_k}} \right)^{\frac{\nu}{r+1}}\\
&\leq c_{\nu, r} \left( \sum_{k=1}^n \frac{1}{1 + \alpha_k} \right)^{-\frac{\nu}{r+1}}.
\end{align*}
Therefore, conditions \eqref{eq:6.4} and \eqref{eq:6.5} of Theorem \ref{thm:6.2} are satisfied with $\vartheta_n = \left( \sum_{k=1}^n \frac{1}{1 + \alpha_k} \right)^{\frac{\nu}{r+1}}$. If $\lim_{k \to \infty}\alpha_k = \alpha \in (0, \infty]$, then  $\beta_n \sim c  \sum_{k=1}^n \frac{1}{1 + \alpha_k}$ for $n \to \infty$ by Lemma \ref{lem:6.2}. Equations \eqref{eq:6.7.1} and \eqref{eq:6.7.3} follow. Eventually, observing that $1 - \frac{\sigma^{r_k +1}}{\sigma^{r_k +1} + \alpha_k} \leq 1 - \frac{\sigma^{r +1}}{\sigma^{r +1} + \alpha_k}$,  equation \eqref{eq:6.7.2} follows instead by a straightforward application of [Lemma 1,2,3 and Theorem~1]\cite{hanke1998nonstationary}.

To prove equation \eqref{eq:6.8} the strategy is the same. We have  $e^{-x} \leq x^{-\nu/(r_n +1)}$ definitely, $1/(\sigma^{r_n +1} + \alpha_k) \sim \alpha_k ^{-1}$ for $n \to \infty$, and hypothesis (ii.2) implies that $\beta_n ^{-1/( r_n +1)}\to 0$ converges to zero.
\end{proof}

%In the case of equation~\eqref{eq:6.7.2}, it follows that
%$\| x^\dagger - x_{\alpha_n, r_n}^n \| = o(n^{-\frac{\nu}{r+1}})$. Furthermore, 
When $r=1$ (classical iterated Tikhonov), equation~\eqref{eq:6.7.2} is the result in \cite[Theorem~1]{hanke1998nonstationary}. On the other hand, if $\lim_{n \to \infty} \alpha_n = \alpha \in (0, \infty]$, then the convergence rate is improved by the small  \nolinebreak ``$o$''.

%\begin{remark}\label{rem:unbound}
%Let $\sup \{ r_k \} < \infty$ and $\lim \alpha_k =0$. If $\alpha_n ^{-1} \leq c \beta_{n-1}$ for a positive real number $c >0$ as in \eqref{eq:6.7.2}, then $\alpha_n ^{-1} = O(q^n)$, where $q= 2\max\{1, c\}$. To have a convergence rate estimate of the error $\| x^\dagger - x_{\alpha_n, r_n}^n \|$ using instead a sequence $\alpha_n ^{-1}$ with a super exponential growth, we can choose a sequence $r_n \uparrow \infty$ such that $\alpha_n = o(\sigma^{r_n +1})$ for every $\sigma>0$. Consequently, we would attain a convergence rate like \eqref{eq:6.8}. For example, choosing $\alpha_n = n^{-pn}$, with $p >1$, and $r_n = n$, then
%$$
%\| x^\dagger - x_{\alpha_n, r_n}^n \| \leq n^{-\nu p}.
%$$
%\end{remark}

\begin{remark}
As we stated in \eqref{eq:6.7.2}, when $\lim_{n \to \infty} \alpha_n
=0$, to obtain a convergence rate of order $O(\beta_n ^{-\nu/(r+1)})$
the sequence $\{ \alpha_n \}$ has to satisfy the condition $\alpha_n
^{-1} \leq c \beta_{n-1}$ for a positive real number $c>0$. Then,
$\sum_{k=1}^n \alpha_k ^{-1} = \beta_n =  O(q^n)$, where $q=(1+c)
>1$. To overcome this bound, in
virtue of (ii.1), (ii.2) of Corollary \ref{cor:6.2}, choosing
sequences $\{ \hat{r}_n \}$ and $\{ \hat{\alpha}_n \}$ such that
$\hat{r}_n$ diverges monotonically and $ \left( \sum_{k=1}^n
\hat{\alpha}_k ^{-1} \right)^{-1} = o(\sigma^{\hat{r}_n +1})$ for
every $0< \sigma \leq 1$, we are able to obtain a faster convergence
rate, in a sense that has still to be defined. In the following
Proposition \ref{prop:6.6} we will give the proof for a specific
case.

%In Section 8 we will compare the standard nonstationary iterated
%Tikhonov method with geometric sequence $\alpha_n = \alpha_0
%q^n$, $q\in(0,1)$, with the NSIWT method with $\alpha_n = 1/n!$ and
%$r_n = n$, showing general faster convergence rate and better
%residual norm for the latter one.
\end{remark}

Following the same approach in \cite[(2.3), (2.4) pag.
26]{brill1987iterative}, we say that the sequence $\{
\hat{x}_n \}$ converges uniformly faster than the sequence
$\{ x_n \}$ if
\begin{equation}\label{eq:6.20}
x^\dagger - \hat{x}_n = R_n (x^\dagger - x_n),
\end{equation}
where $\{ R_n \}$ is a sequence of operators such that $\| R_n \|
\to 0$ as $n \to \infty$. We say instead that $\{ \hat{x}_n \}$
converges non-uniformly faster than $\{ x_n \}$ if \eqref{eq:6.20}
holds and
$$
\inf_{n \in \N} \| R_n \| > 0, \qquad \lim_{n\to \infty} \| R_n x \|
=0  \; \mbox{ for every } x \in \X.
$$
We are ready to state the following comparison result.

\begin{proposition}\label{prop:6.6}
Let $\{x_{\alpha_n}^n\}$ be the sequence generated by the nonstationary iterated Tikhonov
with $\alpha_n = \alpha_0 q^n$, where $\alpha_0 \in (0, \infty), q \in (0,1)$, and let
$\{ x_{\hat{\alpha}_n, \hat{r}_n}^n \}$ be the sequence generated by NSIWT, where
$\hat{\alpha}_n = 1/n!$ and $\hat{r}_n = n$, both applied to the
same compact operator $K : \X \to \Y$. Then, $\{x_{\hat{\alpha}_n,
\hat{r}_n}^n  \}$ converges, non uniformly, faster than $\{
x_{\alpha_n }^n \}$.
\end{proposition}
\begin{proof}
Observe that the sequence $\{x_{\alpha_n}^n\}$ corresponds to a
NSIWT method $\{x_{\alpha_n, r_n}^n\}$ with $r_n =1$ for every $n$.
Moreover, both the sequences $\{x_{\alpha_n}^n\}$ and
$\{x_{\hat{\alpha}_n, \hat{r}_n}^n  \}$ converge, indeed they
satisfy conditions (1) and (2) of Corollary \ref{cor:6.1},
respectively. Assuming that $x_0 =0$ and applying the same strategy
used in Theorem \ref{thm:6.1}, without any effort it is possible to show
that
\begin{align*}
&x^\dagger - x_{\hat{\alpha}_n, \hat{r}_n}^n = \prod_{k=1}^n \hat{\alpha}_k \left( (K^* K)^{\frac{\hat{r}_k +1}{2}} + \hat{\alpha}_k I \right)^{-1} x^\dagger,\\
&x^\dagger = \prod_{k=1}^n \alpha_k^{-1} \left( K^* K + \alpha_k I \right) (x^\dagger - x_{\alpha_n}^n).
\end{align*}
Therefore we find
$$
x^\dagger - x_{\hat{\alpha}_n, \hat{r}_n}^n = \left[ \prod_{k=1}^n
\hat{\alpha}_k \alpha_k^{-1}\left( (K^* K)^{\frac{\hat{r}_k +1}{2}}
+ \hat{\alpha}_k I \right)^{-1}\left( K^* K + \alpha_k I
\right)\right] (x^\dagger - x_{\alpha_n}^n) = R_n (x^\dagger -
x_{\alpha_n}^n).
$$
Since $0 \in \sigma(K^* K)$, we infer $\| R_n \| > 1$ for every $n$, and hence $\inf_{n \in \N} \| R_n \|\ge 1$. If we prove that
$$
\lim_{n \to \infty} \| R_n x\| =0,
$$
for every $x \in \X$, then the thesis follows. Since
$$
\lim_{n \to \infty} \| R_n x\| =0 \Longleftrightarrow \lim_{n\to
\infty} \prod_{k=1}^n \frac{ \hat{\alpha}_k  (\sigma^2 +
\alpha_k)}{\alpha_k(\sigma^{\hat{r}_k +1} + \hat{\alpha}_k)} =0
\Longleftrightarrow  \sum_{k=1}^\infty \frac{\alpha_k
\sigma^{\hat{r}_k +1} - \hat{\alpha}_k \sigma^2}{\alpha_k
\sigma^{\hat{r}_k +1} + \alpha_k \hat{\alpha}_k} = \infty \; \;
\forall \sigma>0,
$$
if we substitute the values $\alpha_n = \alpha_0 q^n$, then
$\hat{\alpha}_n = 1/n!$ and $\hat{r}_n=n$, we obtain
$$
\sum_{k=1}^\infty \frac{\alpha_k \sigma^{\hat{r}_k +1} -
\hat{\alpha}_k \sigma^2}{\alpha_k \sigma^{\hat{r}_k +1} + \alpha_k
\hat{\alpha}_k} = \sum_{k=1}^\infty \frac{1 - \frac{\sigma}{\alpha_0
n! (q\sigma)^n}}{1 + \frac{1/n!}{\sigma^{n+1}}},
$$
and the right hand side of the above equality diverges: indeed
$$
\frac{1 - \frac{\sigma}{\alpha_0 n! (q\sigma)^n}}{1 + \frac{1/n!}{\sigma^{n+1}}} \longrightarrow 1 \; \; \mbox{for every fixed } q, \sigma \in (0,1) \mbox{ and } \alpha_0 \in (0, \infty).
$$
\end{proof}

\subsection{Analysis of convergence for perturbed data}

Let now consider $y^\delta = y + \delta \eta$, with $y \in R(K)$ and $\| \eta \|=1$, i.e., $\| y^\delta - y \| = \delta$. We are concerned about the convergence of the NSIWT method when the initial datum $y$ is perturbed. Hereafter we will use the notation $x_{\alpha_n, r_n}^{n, \delta}$ for the solution of NSIWT \eqref{eq:6.2} with initial datum $y^\delta$.

The following result can be proved similarly to Theorem~1.7 in \cite{brill1987iterative}. 
\begin{theorem}\label{thm:8.1.1}
%Let $\alpha_n$ and $r_n$ be two sequences such that 
Under the assumptions of Corollary \ref{cor:6.1}, if $\{\delta_n\}$ is a sequence convergent to $0$ with $\delta_n \geq 0$ and such that 
\begin{equation}\label{eq:6.9}
\lim_{n\to \infty} \delta_n \cdot \sum_{k=1}^n \alpha_k ^{-1} = 0,
\end{equation}
then, $\lim_{n \to \infty} \| x^\dagger - x_{\alpha_n, r_n} ^{n, \delta_n} \| =0$.
\end{theorem}
\begin{proof}
From the definition of the method \eqref{eq:6.1}, for every given $j, n$, we find that
\begin{eqnarray*}
x_{\alpha_j, r_j}^{j, \delta_n} &= &\left[\left(K^* K  \right)^{\frac{r_j + 1}{2}} + \alpha_j I  \right]^{-1} \left( \left(K^*K\right)^{\frac{r_j - 1}{2}} K^* y^{\delta_n} + \alpha_j x_{\alpha_{j-1}, r_{j-1}}^{j-1, \delta_n} \right)\\
&=& \left\{ I - \alpha_j\left[\left(K^* K  \right)^{\frac{r_j + 1}{2}} + \alpha_j I  \right]^{-1}   \right\} x^\dagger + 
\alpha_j  \left[\left(K^* K  \right)^{\frac{r_j + 1}{2}} + \alpha_j I  \right]^{-1} x_{\alpha_{j-1}, r_{j-1}}^{j-1, \delta_n} \\
&&+ \left[\left(K^* K  \right)^{\frac{r_j + 1}{2}} + \alpha_j I  \right]^{-1} \left(K^* K\right)^{\frac{r_j -1}{2}}K^* (y^{\delta_n} - y),
\end{eqnarray*}
namely,
\begin{eqnarray*}
x^\dagger - x_{\alpha_j, r_j}^{j, \delta_n} & = & \alpha_j\left[\left(K^* K  \right)^{\frac{r_j + 1}{2}} + \alpha_j I  \right]^{-1} (x^\dagger - x_{\alpha_{j-1}, r_{j-1}}^{j-1, \delta_n}) \\
 && -  \left[\left(K^* K  \right)^{\frac{r_j + 1}{2}} + \alpha_j I  \right]^{-1}\left(K^* K\right)^{\frac{r_j -1}{2}} K^* (y^{\delta_n} - y).
\end{eqnarray*}
Hence, by induction, for every fixed $n$ we have
\begin{eqnarray*}
x^\dagger - x_{\alpha_n, r_n}^{n, \delta_n} &=& \prod_{k=1}^n \alpha_k \left[\left(K^* K  \right)^{\frac{r_k + 1}{2}} + \alpha_k I  \right]^{-1} x^\dagger \\
&&- \sum_{k=1}^n \alpha_k ^{-1} \prod_{i=k}^n \alpha_i \left[\left(K^* K  \right)^{\frac{r_i + 1}{2}} + \alpha_i I  \right]^{-1} \left(K^* K\right)^{\frac{r_k -1}{2}} K^* (y^{\delta_n}-y).
\end{eqnarray*}
If we set $g_{k,n}(K^* K)=  \prod_{i=k}^n \alpha_i \left[\left(K^* K  \right)^{\frac{r_i + 1}{2}} + \alpha_i I  \right]^{-1} \left(K^* K\right)^{\frac{r_k -1}{2}}$, then we have 
\begin{align*}
\| g_{k,n}(K^* K) K^* y \|^2 &= \langle g_{k,n}(K^* K) K^* y, g_{k,n}(K^* K)K^* y \rangle \\
							 &=  \langle g_{k,n}(K K^* ) K K^* y, g_{k,n}(K K^* ) y \rangle \\
							 &=  \langle g_{k,n}(K K^*)(K K^*)^{1/2} y, g_{k,n}(K^* K)(KK^*)^{1/2} y \rangle \\
							 &= \|  g_{k,n}(K K^*)(K K^*)^{1/2} y \|^2,
\end{align*}
where we used the fact that $ g_{k,n}(K^* K) K^* =  K^* g_{k,n}(KK^*)$ and that for every bounded Borel function $f$ and $h$,  the product $f(A)h(B)$ commutes if the self-adjoint operators $A$ and $B$ commute \cite[see 12.24]{rudin1991functional}. Therefore,
\begin{align*}
\left\|\prod_{j=k}^n \alpha_j \left[\left(K^* K  \right)^{\frac{r_j + 1}{2}} + \alpha_j I  \right]^{-1} \left(K^* K\right)^{\frac{r_k -1}{2}} K^*  \right\| &= \left\| \prod_{j=k}^n \alpha_j \left[\left(KK^*  \right)^{\frac{r_j + 1}{2}} + \alpha_j I  \right]^{-1} \left(KK^* \right)^{\frac{r_k}{2}}  \right\|\\
&= \max_{\sigma \in [0,1]} \left| \sigma^{r_k} \prod_{j=k}^n \frac{\alpha_j}{\sigma^{r_j + 1} + \alpha_j }  \right| \leq 1.
\end{align*}

It follows that
\begin{align*}
\| x^\dagger - x_{\alpha_n, r_n}^{n, \delta_n} \| &\leq \| \prod_{k=1}^n \alpha_k \left[\left(K^* K  \right)^{\frac{r_k + 1}{2}} + \alpha_k I  \right]^{-1} x^\dagger \| + \sum_{k=1}^n \alpha_k ^{-1} \|y^{\delta_n} - y\| \\
&= \| x^\dagger -   x_{\alpha_n, r_n}^{n} \| + \delta_n \sum_{k=1}^n \alpha_k ^{-1},
\end{align*}
and by Corollary \ref{cor:6.2} and \eqref{eq:6.9}, $\|x^\dagger - x_{\alpha_n, r_n}^{n, \delta_n} \| \to 0$ for $n \to \infty$.
\end{proof}

%\begin{theorem}\cite[Thm 3]{hanke1998nonstationary}
%Let $\alpha_n$ and $r_n$ be two sequences such that Corollary \ref{cor:6.1} holds. If $x^\dagger \in R((K^* K)^{\frac{\nu}{2}}) \cap D((K^* K)^{1/2})$ and $n= n(\delta)$ is chosen by discrepancy principle, then 
%$$
%\| x_{\alpha_{n(\delta)}, r_{n(\delta)}}^{n(\delta), \delta} - x^\dagger \| = O \left(\delta^{\frac{2\nu +1 - r_{n(\delta)}}{2\nu +2}}\right).
%$$
%\end{theorem}

%--------------------------------------------------------------------------
\section{Nonstationary iterated fractional Tikhonov} \label{sec:NSIFT} %label n. 9
%--------------------------------------------------------------------------

\begin{definition}[Nonstationary iterated fractional Tikhonov]
Let $\{ \alpha_n \}_{n\in \N}$ and $\{\gamma_n \}_{n\in \N}$ be sequences of real numbers such that $\alpha_n >0$ and $\gamma_n \geq 1/2$ for every $n$. We define the \emph{nonstationary iterated fractional Tikhonov method} \emph{(NSIFT)} as
\begin{equation}\label{eq:9.1}
\begin{cases}
x_{\alpha_0, \gamma_0}^0 := 0; \\
\left( K^* K + \alpha_n I  \right)^{\gamma_n} x_{\alpha_n, \gamma_n}^n := (K^* K)^{\gamma_n -1} K^* y + \left[\left( K^*K + \alpha_n I \right)^{\gamma_n} - \left( K^*K \right)^{\gamma_n}  \right] x_{\alpha_{n-1}, \gamma_{n-1}}^{n-1}.
\end{cases}
\end{equation}
We denote by $x_{\alpha_n, \gamma_n}^{n, \delta}$ the $n$-th iteration of NSIFT if $y = y^\delta$. 
\end{definition}

\begin{theorem}
The \emph{NSIFT} method \eqref{eq:9.1} converges to $x^\dagger \in \X$ 
as $n \to \infty$ if and only if $\sum_n \left(\frac{\sigma^2}{\sigma^2 + \alpha_n} \right)^{\gamma_n}$ diverges for every $\sigma >0$.
\end{theorem}
\begin{proof}
The proof follows the same steps as  in Theorem \ref{thm:6.1}. Therefore we will omit details. What follows is that
$$
x^\dagger - x_{\alpha_n, \gamma_n}^n = \prod_{k=1}^n \left(K^*K + \alpha_k I  \right)^{-\gamma_k} \left[\left(K^*K + \alpha_k I  \right)^{\gamma_k} - \left(K^*K\right)^{\gamma_k}  \right] x^\dagger,
$$
and hence
$$
\| x^\dagger - x_{\alpha_n, \gamma_n}^n \|^2 = \int_{\sigma(K^*K)} \left| \prod_{k=1}^n \frac{(\sigma^2 + \alpha_k)^{\gamma_k} - \sigma^{2\gamma_k}}{(\sigma^2 + \alpha_k)^{\gamma_k}}  \right|^2 \, d\langle E_{\sigma^2}x^\dagger, x^\dagger \rangle.
$$
Then, the method converges if and only if 
$$
\lim_{n\to \infty} \prod_{k=1}^n \left[1 - \left(\frac{\sigma^2}{\sigma^2 + \alpha_k}\right)^{\gamma_k}  \right]=0
$$
for every $\sigma >0$.
%, namely, if and only if 
%\begin{equation}
%\sum_{k=1}^\infty \left(\frac{\sigma^2}{\sigma^2 + \alpha_k} \right)^{\gamma_k} = \infty
%\end{equation}
%for every $\sigma >0$.
The thesis follows by Lemma~\ref{lem:6.1}. 
\end{proof}

\begin{corollary}\label{cor:9.1}
\begin{itemize}
\item[]
\item[(1)] Let $\lim_{k \to \infty}\gamma_k = \gamma \in [1/2, \infty)$. Then the NSIFT method converges if and only if 
$$
\sum_{k=1}^n \alpha_k ^{-\gamma} = \infty.
$$
More in general, if $\sup_{k \in \N}\{\gamma_k\} =s \in [1/2, \infty)$ and $\sum_{k=1}^\infty \alpha_k ^{-s} = \infty$, then the NSIFT method converges.
\item[(2)]  Let $\lim_{k\to \infty} \gamma_k = \infty$. If $\lim_{k\to \infty}\alpha_k =0$ and $\lim_{k\to \infty}\alpha_k \gamma_k = l \in [0, \infty)$, then the NSIFT method converges.
\end{itemize}
\end{corollary}
\begin{proof}
(1) It is immediate noticing that 
\begin{align*} 
&\sum_{k=1}^n \left(\frac{\sigma^2}{\sigma^2 + \alpha_k}  \right)^{\gamma_k} \sim c \sum_{k=1}^n \left(\frac{\sigma^2}{\sigma^2 + \alpha_k}  \right)^{\gamma} \\
&\sum_{k=1}^n \left(\frac{\sigma^2}{\sigma^2 + \alpha_k}  \right)^{\gamma_k} \geq \sum_{k=1}^n \left(\frac{\sigma^2}{\sigma^2 + \alpha_k}  \right)^{s}.
\end{align*}

(2) We observe that
$$
\left(\frac{\sigma^2}{\sigma^2 + \alpha_k}  \right)^{\gamma_k} = \left(1 - \frac{\alpha_k}{\sigma^2 + \alpha_k}  \right)^{\gamma_k} \sim \textrm{e}^{-\frac{\alpha_k \gamma_k}{\sigma^2 + \alpha_k}} \to \textrm{e}^{-l/\sigma^2} \neq 0
$$
for $k \to \infty$. Then $\sum_{k=1}^n \left(\frac{\sigma^2}{\sigma^2 + \alpha_k}  \right)^{\gamma_k}$ diverges for every $\sigma>0$ and the NSIFT method converges.
\end{proof}

\begin{theorem}\label{thm:9.2}
Let $\{x_{\alpha_n, \gamma_n}^n \}_{n \in \N}$ be a convergent sequence of the NSIFT method, with $x^\dagger \in \X_{\nu}$ for some $\nu >0$, and let $\{ \vartheta_n \}_{n \in \N}$ be a divergent sequence of positive real numbers. If
\begin{subequations}
\begin{equation} 
\lim_{n\to \infty} \vartheta_n \sigma^\nu \prod_{k=1}^n \left( 1 - \frac{\sigma^{2\gamma_k}}{\left( \sigma^{2} + \alpha_k \right)^{\gamma_k}}  \right)=0 \qquad \mbox{for every } \sigma \in \sigma(K)\setminus \{0\}; \label{eq:9.4}
\end{equation}
\begin{equation}
\sup_{\sigma \in  \sigma(K)\setminus \{0\}} \vartheta_n \sigma^\nu \prod_{k=1}^n \left( 1 - \frac{\sigma^{2\gamma_k}}{\left( \sigma^{2} + \alpha_k \right)^{\gamma_k} }  \right) \leq c < \infty \qquad \mbox{ uniformly with respect to } n, \label{eq:9.5}
\end{equation}
\end{subequations}
then 
\begin{equation}
\| x^\dagger - x_{\alpha_n, \gamma_n}^n \| = o (\vartheta_n ^{-1}).
\end{equation}
\end{theorem}
\begin{proof}
As seen in Theorem \ref{thm:6.2}, the thesis follows easily from the Dominated Convergence Theorem.
\end{proof}

\begin{corollary}
Let $\{ \gamma_k \}_{k \in \N}$ be a sequence of positive real numbers, $\gamma_k \geq 1/2$, and let $x^\dagger \in \X_{\nu}$ for some $\nu >0$. If
\begin{itemize}
\item[(i.1)] $\sup_{k \in \N} \{ \gamma_k \} = s \in [1/2, \infty)$,
\item[(i.2)] $\lim_{n \to \infty} \beta_n = \infty$,
\end{itemize}
then
\begin{align}
&\| x^\dagger - x_{\alpha_n, \gamma_n}^n \| = o (\beta_n ^{- \frac{\nu}{2s}}) \qquad \mbox{if } \, \exists\lim_{k\to \infty}\alpha_k = \alpha \in (0, \infty],\\
&\| x^\dagger - x_{\alpha_n, \gamma_n}^n \| = o (\tilde{\beta}_n ^{- \frac{\nu}{2s}}) \qquad \mbox{otherwise},
\end{align}
where we defined
$$
\beta_n = \sum_{k=1}^n \alpha_k ^{-s}, \qquad \tilde{\beta}_n =\sum_{k=1}^n \frac{1}{1+ \alpha_k ^s}.
$$
On the contrary, if
\begin{itemize}
\item[(ii.1)] $\lim_{k \to \infty}\gamma_k = \infty$,
\item[(ii.2)] $\lim_{k\to \infty}\alpha_k =0$ and $\lim_{k\to \infty}\alpha_k \gamma_k =0$,
\end{itemize}
then
\begin{equation}
\| x^\dagger - x_{\alpha_n, \gamma_n}^n \| = o (n^{-1}).
\end{equation}
\end{corollary}
\begin{proof}
See Corollary \ref{cor:6.2}. In particular, for the second statement we use the fact that
$$
\textrm{e}^{-\sum_{k=1}^n\left( \frac{\sigma^2}{\sigma^2 + \alpha_k}\right)^{\gamma_k}} 
%\sim c \textrm{e}^{-\sum_{k=1}^n 1} 
= o (n^{-1}).
$$
\end{proof}

%\subsection{Analysis of convergence for perturbed data}
\begin{theorem}
Under the assumptions of Corollary \ref{cor:9.1}, if $\{\delta_n\}$ is a sequence convergent to $0$ with $\delta_n \geq 0$ and such that 
\begin{equation}\label{eq:9.9}
\lim_{n\to \infty} \delta_n \cdot \sum_{k=1}^n \alpha_k ^{-\gamma_k} = 0,
\end{equation}
then, $\lim_{n \to \infty} \| x^\dagger - x_{\alpha_n, \gamma_n} ^{n, \delta_n} \| =0$.
\end{theorem}
\begin{proof}
Here is a sketch of the proof, since it follows step by step from the proof of Theorem \ref{thm:8.1.1}. If we set 
\begin{align*}
&\psi_k (K^* K) := \left[ (K^* K + \alpha_k I)^{\gamma_k} - (K^* K)^{\gamma_k} \right]\\
&\phi_k (K^* K) := \psi_k (K^* K) \left[K^* K  + \alpha_k I  \right]^{-\gamma_k},
\end{align*}
then from \eqref{eq:9.1} it is possible to show that
\begin{align*}
x^\dagger - x_{\alpha_n, \gamma_n}^{n, \delta_n} = \prod_{k=1}^n \phi_k (K^* K) x^\dagger 
- \sum_{k=1}^n \psi_k (K^* K)^{-1} \prod_{i=k}^n \phi_i (K^* K) \left(K^* K\right)^{\gamma_k -1} K^* (y^{\delta_n}-y),
\end{align*}
for every integer $n$ and for every perturbed data $y^{\delta_n} = y + \delta_n \eta$. Owing to the equality
$$
\left\| \prod_{i=k}^n \phi_i (K^* K) \left(K^* K\right)^{\gamma_k -1} K^*  \right\| = \left\|  \prod_{i=k}^n \phi_i (KK^*) \left(KK^*\right)^{\gamma_k -1} (KK^*)^{1/2}  \right\|,
$$
we deduce
\begin{align*}
\| x^\dagger - x_{\alpha_n, \gamma_n}^{n, \delta_n} \| &\leq \|  x^\dagger - x_{\alpha_n, \gamma_n}^{n}  \| + \delta_n \sum_{k=1} ^n \left\| \psi_k (K^* K)^{-1}  \right\| \\
&=  \|  x^\dagger - x_{\alpha_n, \gamma_n}^{n}  \| + \delta_n \sum_{k=1} ^n \alpha_k ^{-\gamma_k}.
\end{align*}
\end{proof}

%----------------------------------------------------------------------------------------------------
\section{Numerical results}\label{sec:numerical}
We now give few selected examples with a special focus on the nonstationary iterations proposed in this paper.
For a larger comparison between fractional and classical Tikhonov refer to
\cite{klann2008regularization,hochstenbach2011fractional,Gert2015fractional}.
To produce our results we used Matlab 8.1.0.604 using a laptop pc with processor Intel iCore i5-3337U with 6 GB of RAM running Windows 8.1.

We add to the noise-free right-hand side vector $y$, the ``noise-vector'' $e$ that has in all
examples normally distributed pseudorandom entries with mean zero, and is
normalized to correspond to a chosen noise-level 
$$\xi=\frac{\|e\|}{\|y\|}.$$

As a stopping criterion for the methods we used the Discrepancy Principle \cite{hankehansen1993}, that terminates the 
iterative method at the iteration
$$
\hat{k} = \min_k\{k:\,\|y^\delta-Kx_k\|\leq\tau\delta\},
$$
where $\tau=1.01$. This criterion stops the iterations when the norm of the residual reaches the norm of the noise so that the latter is not reconstructed.

To compare the restorations with the different methods, we consider both the visual representation and the relative restoration error that is $\|\hat{x}-x^\dag\|/\|x^\dag\|$ for the computed approximation \nolinebreak $\hat{x}$.

\begin{figure}
\centering
\begin{tabular}{cc}
\includegraphics[width=0.4\textwidth,clip=true,trim=1.6cm 0cm 1.6cm 0cm]{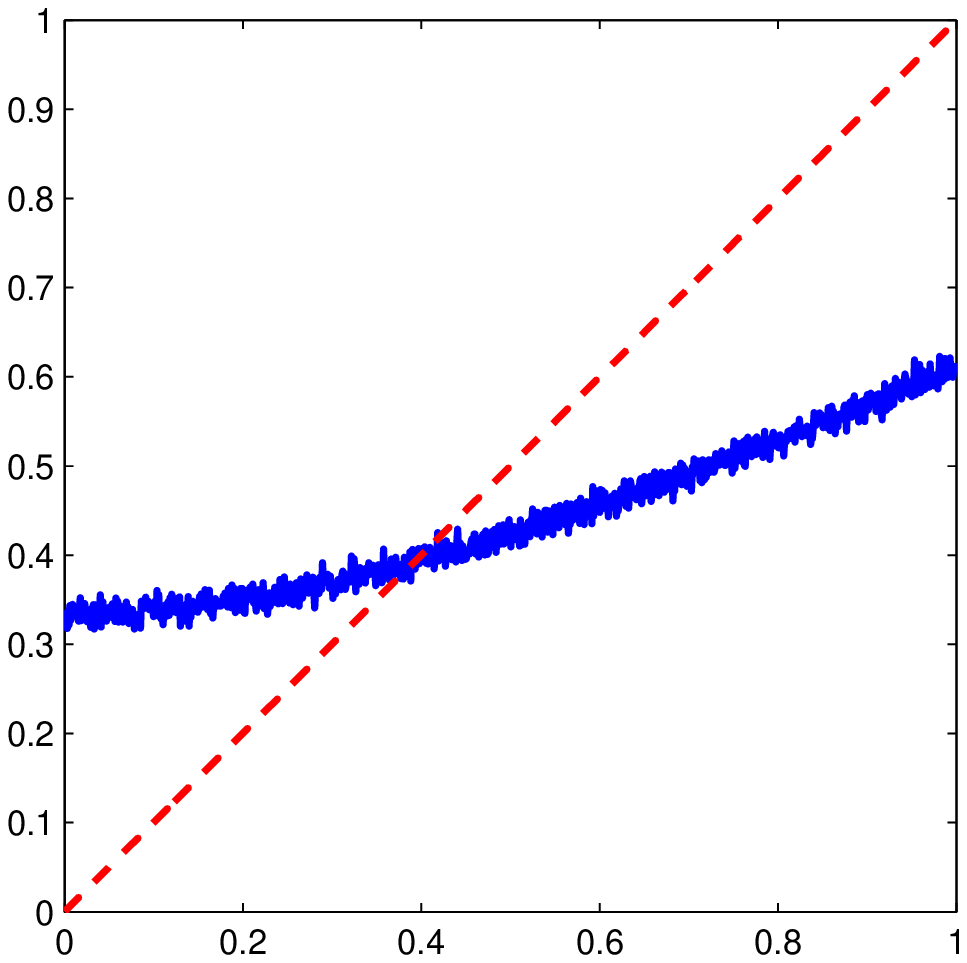} & 
\includegraphics[width=0.4\textwidth]{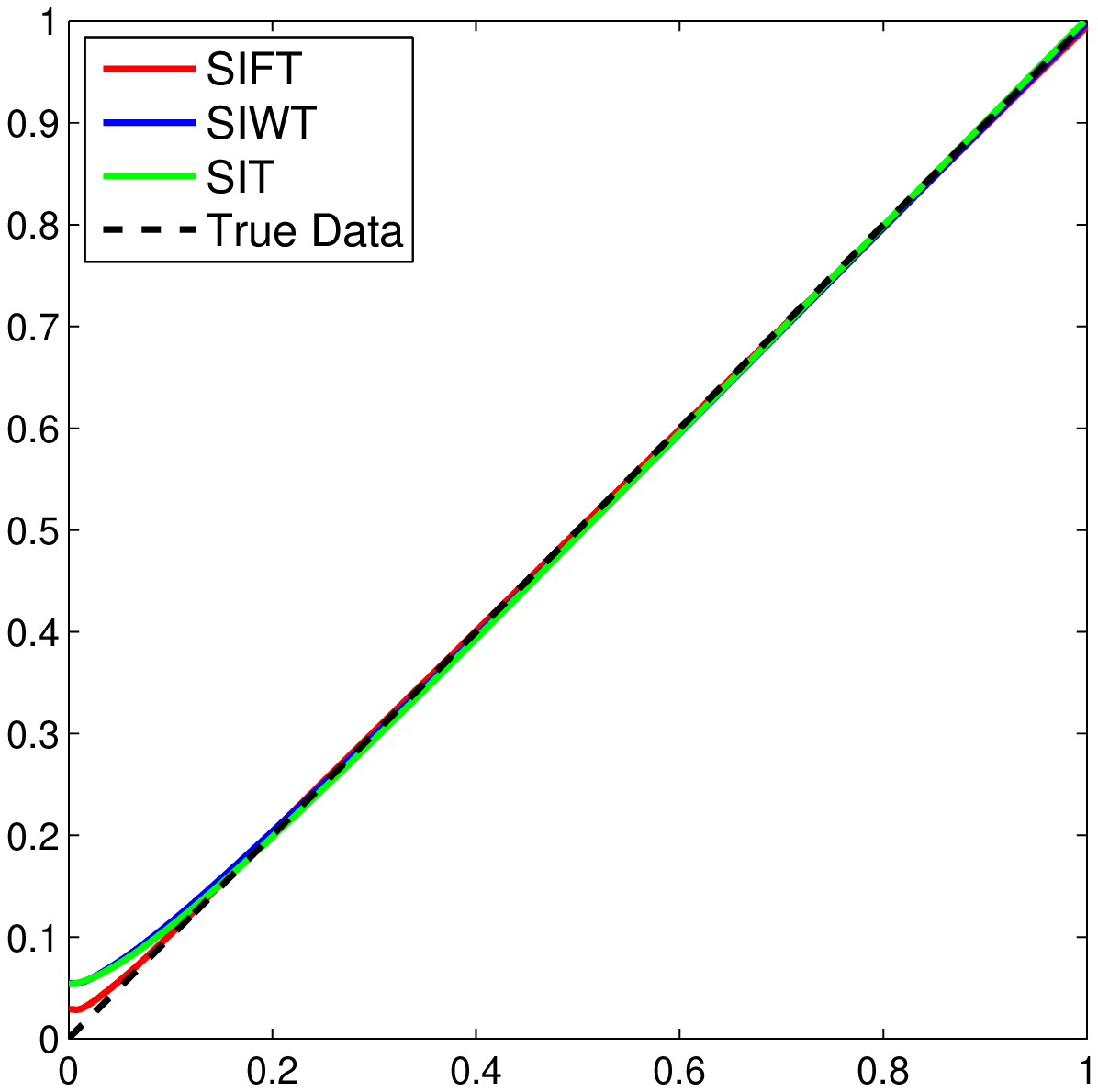}\\
(a) & (b)
\end{tabular}
\caption{Example 1 -- ``Foxgood'' test case: (a)  the true solution (dashed curve) and the observed data (solid curve), 
(b) approximated solutions by SIFT with $\gamma=0.8$ and $\alpha=10^{-3}$,
							SIWT with $r=0.6$ and $\alpha=10^{-2}$, and SIWT with $r=1$ and $\alpha=10^{-3}$. }
\label{fig:Foxgood}
\end{figure}

\subsection{Example 1}
This test case is the so-called \emph{Foxgood} in the toolbox \textsc{Regularization tool} by P. Hansen \cite{RegularizationTools} using $1024$ points. 
We have added a noise vector with $\xi=0.02$ to the observed signal. In Figure~\ref{fig:Foxgood}(a) the true signal and the measured data can be seen.

In  Table \ref{tbl:Foxgood_st} we show the relative errors with different choices of $\alpha$, $r$ and $\gamma$. In brackets we report the iteration at which the discrepancy principle stopped the method.
Note that SIFT with $\gamma=1$ and SIWT with $r=1$ are exactly the classical Tikhonov method and hence produce the same result.
Figure \ref{fig:Foxgood}(b) shows the reconstruction for SIFT with $\gamma=0.8$ and $\alpha=10^{-3}$,
SIWT with $r=0.6$ and $\alpha=10^{-2}$, and 
SIWT with $r=1$ (classical Iterated Tikhonov) with $\alpha=10^{-3}$.

\begin{table}
%\resizebox{\textwidth}{!}{
\begin{center}
\small
\begin{tabular}{l|l|lllll}
\multirow{2}{*}{$\alpha$}&\multirow{2}{*}{Method}&\multicolumn{5}{c}{$r/\gamma$}\\\cline{3-7}
                                     &&0.4        &0.6         &0.8         &1           &1.2\\
															   \hline								 
\multirow{2}{*}{$5\times10^{-2}$}&SIFT&337.09(7)  &0.02498(13)&0.03481(19)&0.03752(29)&0.03838(43)\\
                                 &SIWT&0.02589(9)&0.03202(13)&0.03609(19) &0.03752(29)&0.03932(43)\\\hline
\multirow{2}{*}{$10^{-2}$}       &SIFT&320.85(3)  &0.02048(5) &0.02633(7) &0.03731(7) &0.03783(9)\\
                                 &SIWT&0.01697(3)&0.01818(5) &0.03361(5) &0.03731(7) &0.03672(11)\\\hline
\multirow{2}{*}{$5\times10^{-3}$}&SIFT&423.37(3)  &0.02216(3) &0.02190(5) &0.03102(5) &0.03723(5)\\
                                 &SIWT&0.02421(3)&0.01573(3) &0.03186(3) &0.03103(5) &0.03347(7)\\\hline
\multirow{2}{*}{$10^{-3}$}       &SIFT&402.97(1)  &0.02299(1)  &\textbf{0.00698(3)}&0.01756(3)  &0.02443(3)\\
                                 &SIWT&0.06403(1)&0.02210(1) &0.02528(1) &0.01756(3)  &0.02736(3)\\\hline
\multirow{2}{*}{$5\times10^{-4}$}&SIFT&531.72(1)  &0.02119(1) &0.01729(1) &0.02507(1) &0.03119(1)\\
                                 &SIWT&0.10518(1) &0.04506(1) &0.01482(1) &0.02507(1) &0.02086(3)\\\hline
\multirow{2}{*}{$10^{-4}$}       &SIFT&1012.2(1)  &0.07246(1) &0.04229(1) &0.02704(1) &0.01675(1)\\
                                 &SIWT&0.25927(1) &0.13000(1)     &0.07213(1) &0.02704(1) &0.01154(1)
\end{tabular}
\end{center}
\caption{Example 1: relative errors for SIWT and SIFT for different choices of $\alpha$, $r$, and $\gamma$.}
\label{tbl:Foxgood_st}
\end{table}

From these results, using both fractional and weighted  iterated Tikhonov, we can see that we can obtain better restorations  than with the classical version. However, in order to obtain such results, one has to evaluate $\alpha$ very carefully. Indeed 
$\alpha$ does not only affects the convergence speed, but also the quality of the restoration: a small perturbation in $\alpha$ can lead to quite different restoration errors. 
The nonstationary version of the methods can help also to avoid such a careful and often difficult estimation. 

For the nonstationary iterations we assume the regularization parameter $\alpha_n$ at each iteration be given according to the geometric sequence
\begin{equation}
\alpha_n=\alpha_0q^n, \qquad q\in (0,1), \qquad n=1,2,\dots.
\label{eq:alpha_nonstat}
\end{equation}
Setting $r_n=0.6$ and $\gamma_n=0.8$, Table \ref{tbl:Foxgood_nst_alpha} shows that NSIFT and NSIWT provide a relative error lower than the classical nonstationary iterated Tikhonov (NSIT). 
\begin{table}
%\resizebox{\textwidth}{!}
\begin{center}
\small
\begin{tabular}{l|l|lll}
\multirow{2}{*}{$\alpha_0$}&\multirow{2}{*}{Method}&\multicolumn{3}{c}{$q$}\\\cline{3-5}
                          &       &0.7           &0.8         &0.9\\\hline
\multirow{5}{*}{$10^{-1}$}&NSIFT ($\gamma_n=0.8$)  &0.024453(9)&0.030868(11)&0.028849(17)	\\
                          &NSIWT ($r_n=0.6$) &0.025223(7)&0.027628(9)&0.028534(13)	\\
                          &NSIT   &0.035162(9)&0.031627(13)&0.036472(19)	\\
&NSIFT ($\gamma_n$ in \eqref{eq:r_nonstat}) &0.032489(9)&0.027974(13)&0.037199(17)	\\
                          &NSIWT ($r_n$ in \eqref{eq:r_nonstat})&0.031493(9)&0.027436(13)&0.036059(17)	\\                          
                          \hline
\multirow{5}{*}{$10^{-2}$}&NSIFT ($\gamma_n=0.8$) &\textbf{0.014781(5)}&0.021687(5)&0.028709(5)	\\
                          &NSIWT ($r_n=0.6$)&\textbf{0.014503(3)}&0.021501(3)&0.028396(3)	\\
                          &NSIT   &0.024838(5)&0.030866(5)&0.028835(7)\\
&NSIFT ($\gamma_n$ in \eqref{eq:r_nonstat})&0.023848(5)&0.030002(5)&0.027636(7)	\\
                          &NSIWT ($r_n$ in \eqref{eq:r_nonstat})&0.023482(5)&0.029638(5)&0.027366(7)	                          
\end{tabular}
\end{center}
\caption{Example 1: relative errors for NSIWT and NSIFT with the nonstationary $\alpha_n$ in \eqref{eq:alpha_nonstat} and different choices of $r_n$ and $\gamma_n$ (NSIT is $r_n=\gamma_n=1$).}
\label{tbl:Foxgood_nst_alpha}
\end{table}
Finally, since NSIFT and NSIWT allow a nonstationary choice also for $r_n$ and $\gamma_n$, 
in Table~\ref{tbl:Foxgood_nst_alpha}  we report the results for the following nonincreasing sequences
\begin{equation}\label{eq:r_nonstat}
r_n=\gamma_n=
\left\{\begin{array}{ll}
1-\frac{n-1}{100}& n<50,\\
\frac{1}{2}&\mbox{otherwise}.
\end{array}\right.
%\gamma_n=\left\{\begin{array}{ll}
%1-\frac{n-1}{100}& n<50\\
%\frac{1}{2}&\mbox{else}
%\end{array}\right.
%\label{eq:gamma_nonstat}
\end{equation}
Again both NSIWT and NSIFT are able to get better results than NSIT. 
Even tough the errors are not as good as those for the best choices $r_n=0.6$ and $\gamma_n=0.8$, the choice
\eqref{eq:r_nonstat} stresses the robustness of our nonstationary iterations.
%\begin{table}
%\resizebox{\textwidth}{!}{
%\begin{tabular}{l|l|lll}
%\multirow{2}{*}{$\alpha_0$}&\multirow{2}{*}{Method}&\multicolumn{3}{c}{$q$}\\\cline{3-5}
%                          &       &0.7          &0.8          &0.9\\\hline
%\multirow{3}{*}{$10^{-1}$}&NSIFT &0.032489(9)&0.027974(13)&0.037199(17)	\\
%                          &NSIWT &0.031493(9)&0.027436(13)&0.036059(17)	\\
%                          &NSIT   &0.035162(9)&0.031627(13)&0.036472(19)	\\\hline
%\multirow{3}{*}{$10^{-2}$}&NSIFT &0.023848(5)&0.030002(5)&0.027636(7)	\\
 %                         &NSIWT &0.023482(5)&0.029638(5)&0.027366(7)	\\
 %                         &NSIT   &0.024838(5)&0.030866(5)&0.028835(7)
%\end{tabular}}
%\caption{Foxgood Errors for the non stationary $\alpha$ and $r,\gamma$ case}
%\label{tbl:Foxgood_nst}
%\end{table}

%------------------------------------------------------------------------------------------------------------------------------------
\subsection{Example 2}
We consider the test problem \emph{deriv2($\cdot$,3)} in the toolbox \textsc{Regularization tool} by P. Hansen \cite{RegularizationTools} using $1024$ points. For the noise vector it holds $\xi=0.05$. In Figure~\ref{fig:deriv}(a) we can see the measured data and the true signal. We compare NSIWT and NSIFT with the NSIT.

\begin{figure}
\centering
\begin{tabular}{cc}
\includegraphics[width=0.45\textwidth]{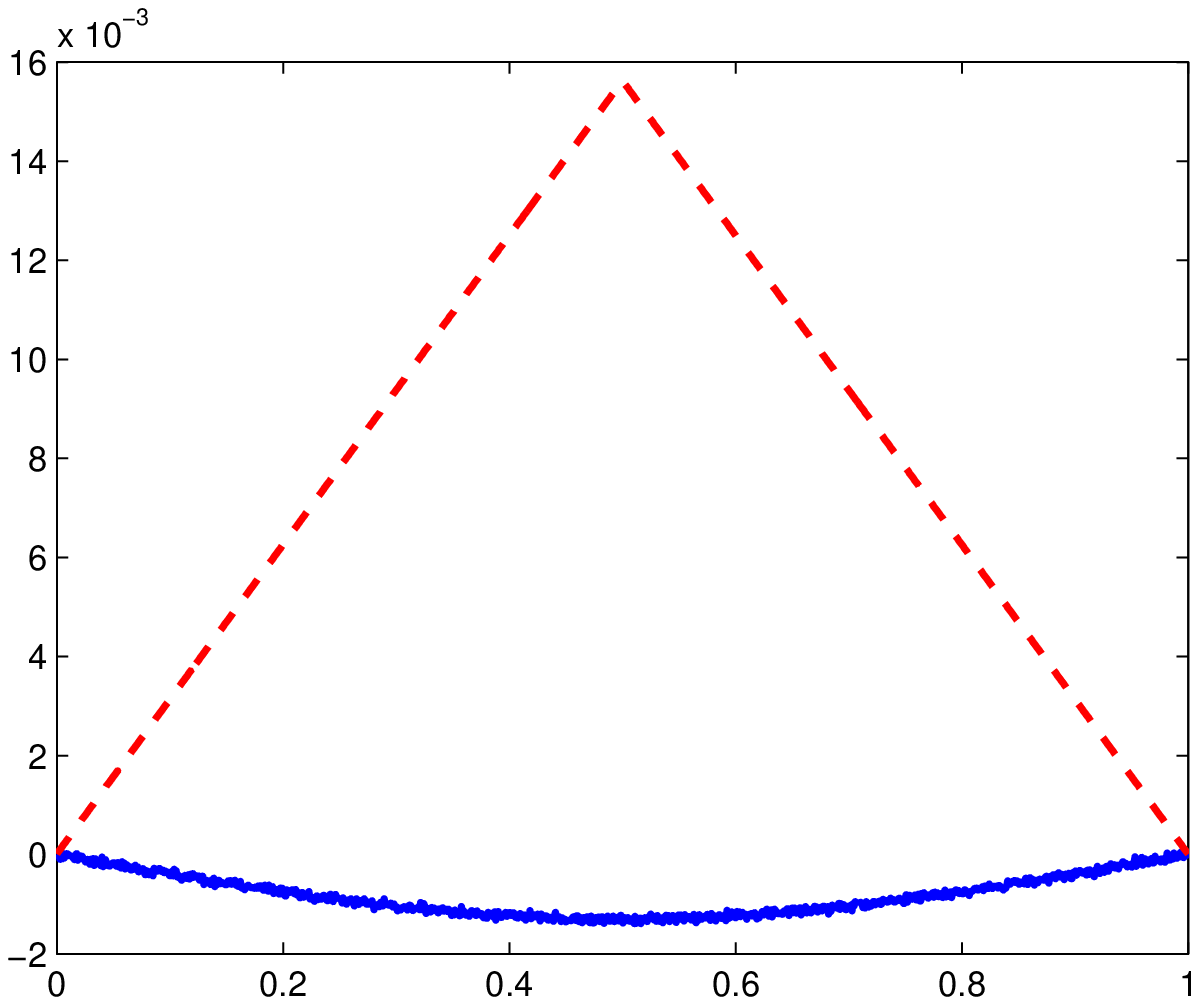} & 
\includegraphics[width=0.45\textwidth]{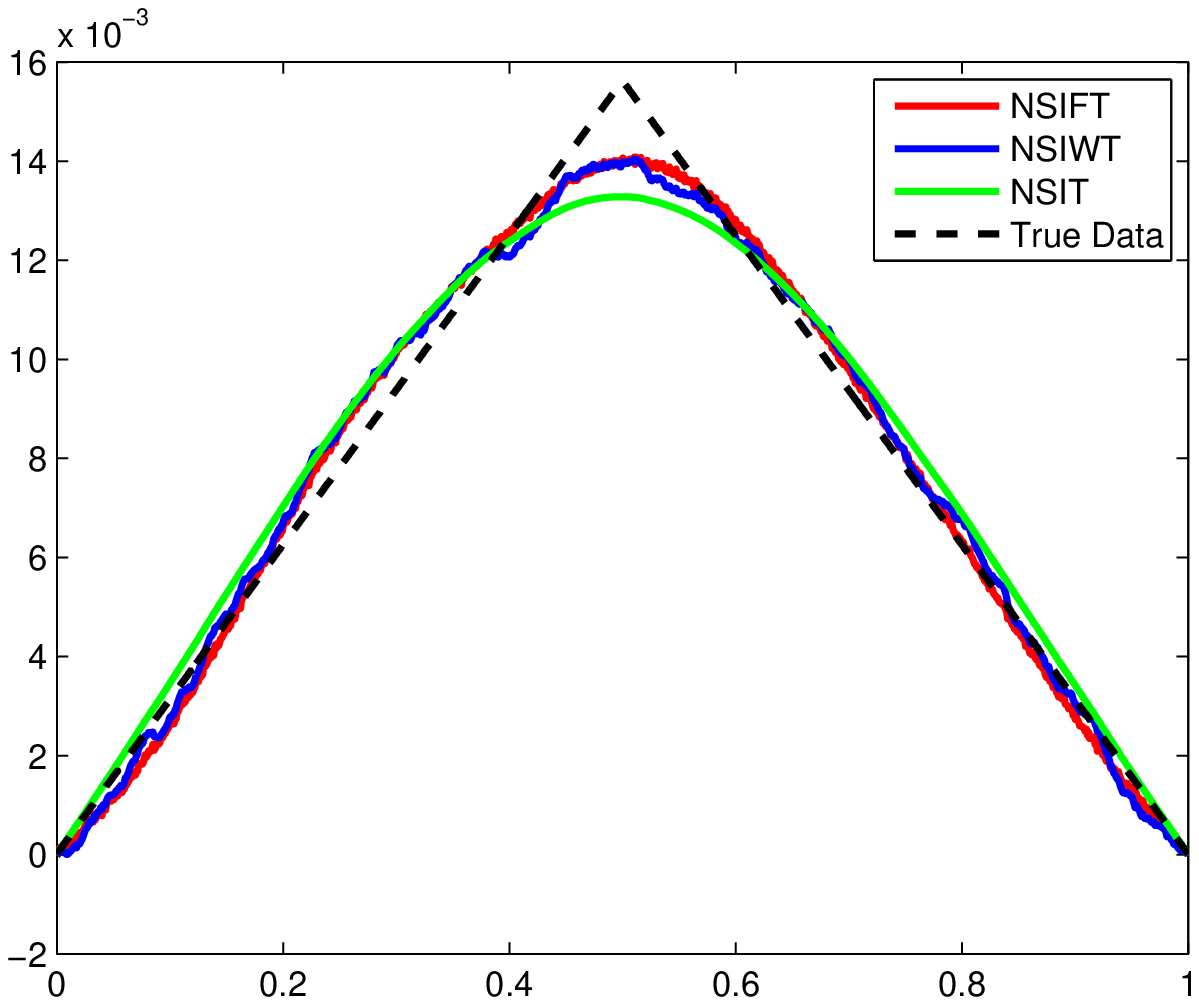}\\
(a) & (b)
\end{tabular}
\caption{Example 2 -- ``deriv2'' test case: (a)  the true solution (dashed curve) and the observed data (solid curve), 
(b) approximated solutions.}
\label{fig:deriv}
\end{figure}

Firstly, $\alpha_n$ is defined by the classical choice in \eqref{eq:alpha_nonstat}. 
Table \ref{tbl:deriv_bounded} shows the results for different choices of $r_n$ and $\gamma_n$. Note that NSIWT and NSIFT usually outperform NSIT. Nevertheless, our nonstationary iterations allow also unbounded sequences of  $r_n$ and $\gamma_n$. Therefore, according to Proposition~\ref{prop:6.6}, we set
\begin{equation}\label{eq:Ns_succession}
\alpha_n=\frac{1}{n!}, \qquad
r_n=\frac{n}{10}, \qquad
\gamma_n=\frac{n}{2}.
\end{equation}
Table \ref{tbl:deriv_unbounded}  shows that the relative restoration error obtained with the unbounded sequences $r_n$ and $\gamma_n$ in \eqref{eq:Ns_succession} is lower than the best one (according to Table~\ref{tbl:deriv_bounded}), obtained by NSIT by employing  the geometric sequence \eqref{eq:alpha_nonstat} for $\alpha_n$. The computed approximations are also compared in Figure~\ref{fig:deriv}(b), where we note a better restoration of the corner for NSIWT and NSIFT.

\begin{table}
%\resizebox{\textwidth}{!}{
\begin{center}
\small
\begin{tabular}{l|l|lll}
\multirow{2}{*}{$\alpha_0$}&\multirow{2}{*}{Method}&\multicolumn{3}{c}{$q$}\\\cline{3-5}
                          &                      &0.7         &0.8         &0.9\\\hline
\multirow{5}{*}{$10^{-1}$}&NSIFT ($\gamma_n=0.8$)&0.08981(11)&0.09394(13)&0.09445(19)\\
                          &NSIWT ($r_n=0.6$)     &0.08051(13) &0.09181(17)&0.09401(29)\\
                          &NSIT                  &0.08502(15) &0.09175(21) &0.09466(37)\\
						  &NSIFT ($\gamma_n$ in \eqref{eq:r_nonstat})&0.09428(13)&0.09089(19)&0.09327(29)\\
						  &NSIWT ($r_n$ in \eqref{eq:r_nonstat})&0.09073(13) &0.08648(19)&0.09199(29)\\\hline
\multirow{5}{*}{$10^{-2}$}&NSIFT ($\gamma_n=0.8$)&0.09114(5)&0.08953(7)&0.08998(9)\\
                          &NSIWT ($r_n=0.6$)     &\textbf{0.07807(7)} &0.09411(7)&0.09183(11)\\
                          &NSIT  &0.08183(9) &0.09174(11)&0.09379(17)\\
						  &NSIFT  ($\gamma_n$ in \eqref{eq:r_nonstat})&\textbf{0.07839(9)}  &0.08721(11)&0.09246(15)\\
						  &NSIWT ($r_n$ in \eqref{eq:r_nonstat}) &0.09399(7) &0.08389(11)&0.08990(15)
\end{tabular}
\end{center}
\caption{Example 2: relative errors for NSIWT and NSIFT with the nonstationary $\alpha_n$ in \eqref{eq:alpha_nonstat}  and different choices of $r_n$ and $\gamma_n$ (NSIT is $r_n=\gamma_n=1$).}
\label{tbl:deriv_bounded}
\end{table}
\begin{table}
\centering
\begin{tabular}{l|ccc}
&NSIFT&NSIWT &NSIT \\ \hline
Error &0.054831(9)&0.059211(7)&0.081835(9)
\end{tabular}
\caption{Example 2: relative restoration errors for NSIFT and NSIWT with 
parameters in \eqref{eq:Ns_succession} and NSIT with $\alpha_n=0.01 \cdot 0.7^n$.}
\label{tbl:deriv_unbounded}
\end{table}

%------------------------------------------------------------------------------------------------------------------------------------
\subsection{Example 3}
We consider the test problem \emph{blur($\cdot$,$\cdot$,$\cdot$)} in the toolbox \textsc{Regularization tool} by P. Hansen \cite{RegularizationTools}. This is a two dimensional deblurring problem, the true solution is a $40\times 40$ image, the blurring operator is a symmetric BTTB (block Toeplitz with Toeplitz block) with bandwidth $6$. This blur is created by a truncated Gaussian point spread function with variance $2$. For the noise vector it holds $\nu=0.005$. Figure~\ref{fig:blur}(a) shows the true image while the observed image is in Figure~\ref{fig:blur}(b). %We compare NSIWT and NSIFT with the NSIT.
\begin{figure}
\centering
\begin{tabular}{cc}
\includegraphics[width=0.30\textwidth]{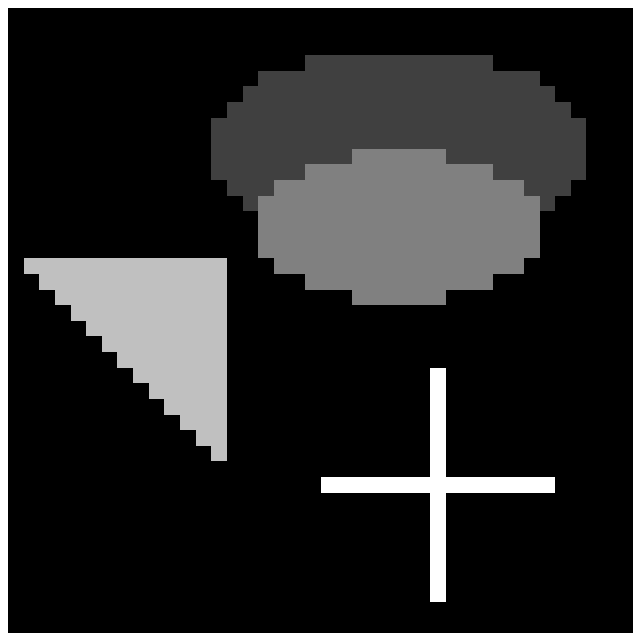} & 
\includegraphics[width=0.30\textwidth]{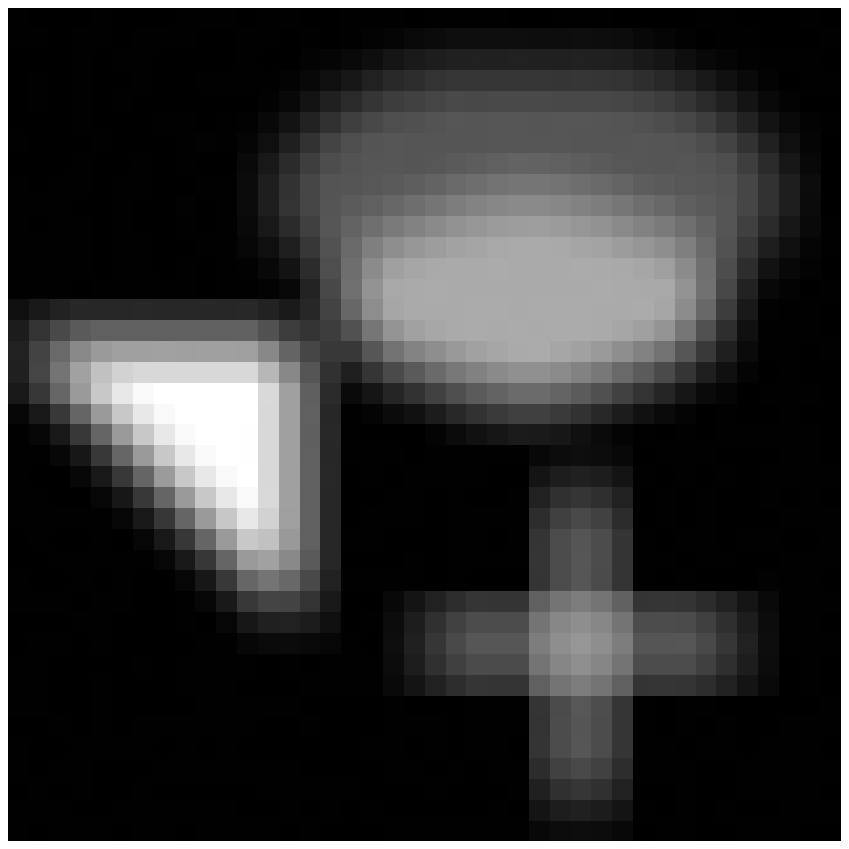}\\
(a) & (b)
\end{tabular}
\caption{Example 3 -- ``blur'' test case: (a)  the true image, 
(b) the measured data.}
\label{fig:blur}
\end{figure}

Firstly, $\alpha_n$ is defined by the classical choice in \eqref{eq:alpha_nonstat}. 
Table \ref{tbl:blur_stat} provides the results for a good stationary choice of $r_n$ and $\gamma_n$. Note that NSIWT and NSIFT usually outperform NSIT. 
Table~\ref{tbl:blur_nstat} shows that the relative restoration error obtained with the unbounded sequences $r_n$ and $\gamma_n$ in \eqref{eq:Ns_succession} is lower than the best one (according to Table~\ref{tbl:blur_stat}), obtained by the stationary choice of $r_n$ and $\gamma_n$.
We note that NSIWT and NSIFT are less sensitive than NSIT to an appropriate choice of $\alpha_0$ and $q$. In particular using  $r_n$ and $\gamma_n$ in \eqref{eq:Ns_succession}, NSIWT and NSIFT do not need any parameter estimation and the computed solutions have a relative restoration error lower than NSIT with the best parameter setting (see Table~\ref{tbl:blur_stat}) and they provide also a better reconstruction, in particular of the edges, see Figure~\ref{fig:blur_rec}.

Finally, note that for the NSIT a nondecreasing sequence of $\alpha_n$ could be considered instead of the geometric sequence \eqref{eq:alpha_nonstat}, see \cite{dona2012}. Nevertheless, this strategy requires a proper choice of $\alpha_0$ and this is out of the scope of this paper, but it could be investigated in the future in connection with our fractional and weighted variants.
A further development of our iterative schemes is in the direction of the nonstationary preconditioning strategy in \cite{DH13}, which is inspired by an approximated solution of the NSIT and hence could be investigated also in a fractional framework.

\begin{table}
%\resizebox{\textwidth}{!}{
\begin{center}
\small
\begin{tabular}{l|l|lll}
\multirow{2}{*}{$\alpha_0$}&\multirow{2}{*}{Method}&\multicolumn{3}{c}{$q$}\\\cline{3-5}
                          &                                           &0.7         &0.8         &0.9\\\hline
\multirow{3}{*}{$10^{-1}$}&NSIFT ($\gamma_n=0.5$) &0.19970(9)   &0.19526(13)  &0.19847(17)\\
                          &NSIWT ($r_n=0.2$)      &0.18936(7)   &\textbf{0.18920(9)}   &0.19732(11)\\
                          &NSIT                   &0.19816(15)  &0.21786(20)  &0.28703(20)\\
\hline                                                                                   
\multirow{3}{*}{$10^{-2}$}&NSIFT ($\gamma_n=0.5$) &0.19398(5)   &0.19962(5)   &0.19595(7)\\
                          &NSIWT ($r_n=0.2$)      &0.20822(3)   &0.19547(3)   &0.19109(3)\\
                          &NSIT                   &0.19518(9)   &0.20531(11)  &0.20747(17)\\
\end{tabular}
\end{center}
\caption{Example 3: relative errors for NSIWT and NSIFT with the nonstationary $\alpha_n$ in \eqref{eq:alpha_nonstat}.}
\label{tbl:blur_stat}
\end{table}
\begin{table}
\centering
\begin{tabular}{l|ccc}
&NSIFT&NSIWT &NSIT \\ \hline
Error &0.19335(10)&0.18765(8)&0.19518(9)
\end{tabular}
\caption{Example 3: relative restorations errors for NSIFT and NSIWT with 
parameters in \eqref{eq:Ns_succession} and NSIT with $\alpha_n=0.01\cdot 0.7^n$.}
\label{tbl:blur_nstat}
\end{table}

\begin{figure}
\centering
\begin{tabular}{ccc}
\includegraphics[width=0.30\textwidth]{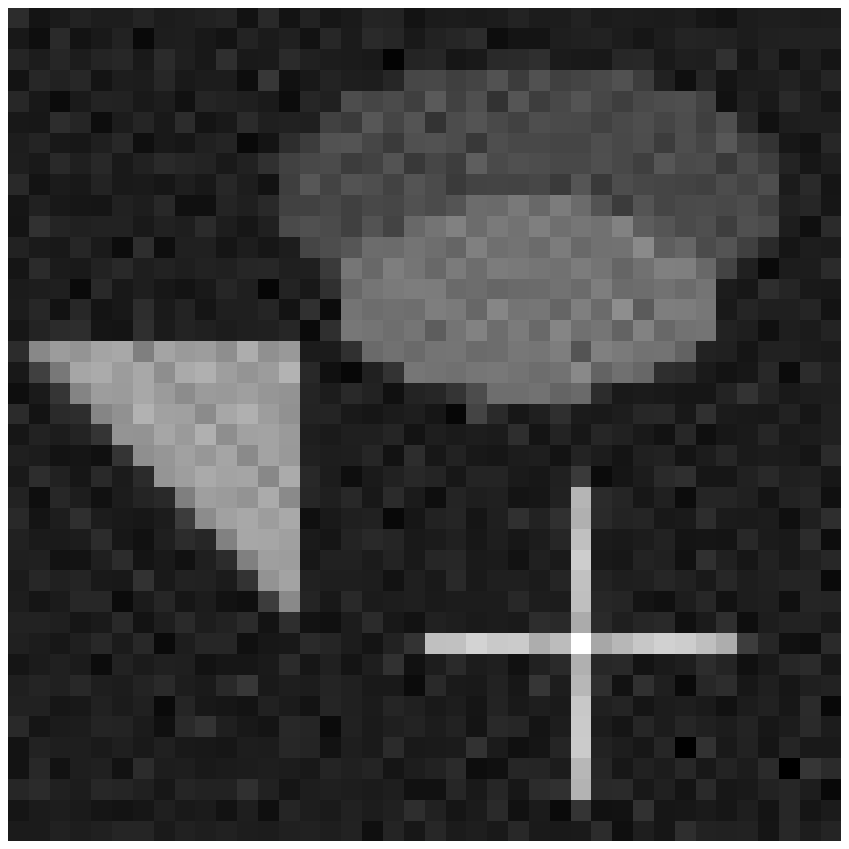} & 
\includegraphics[width=0.30\textwidth]{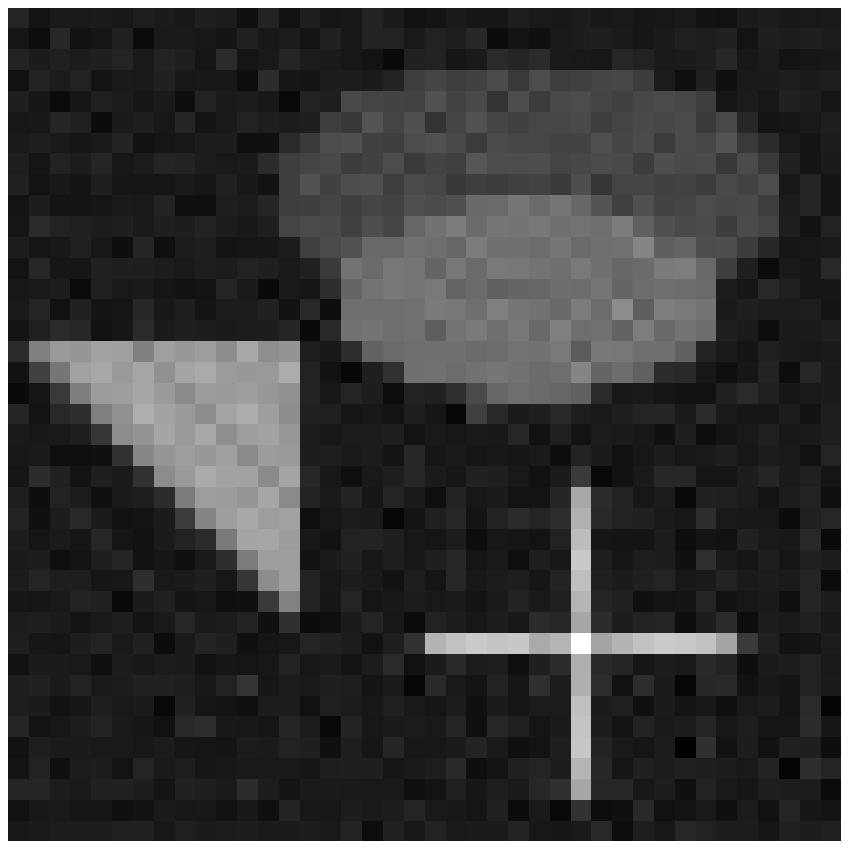} & 
\includegraphics[width=0.30\textwidth]{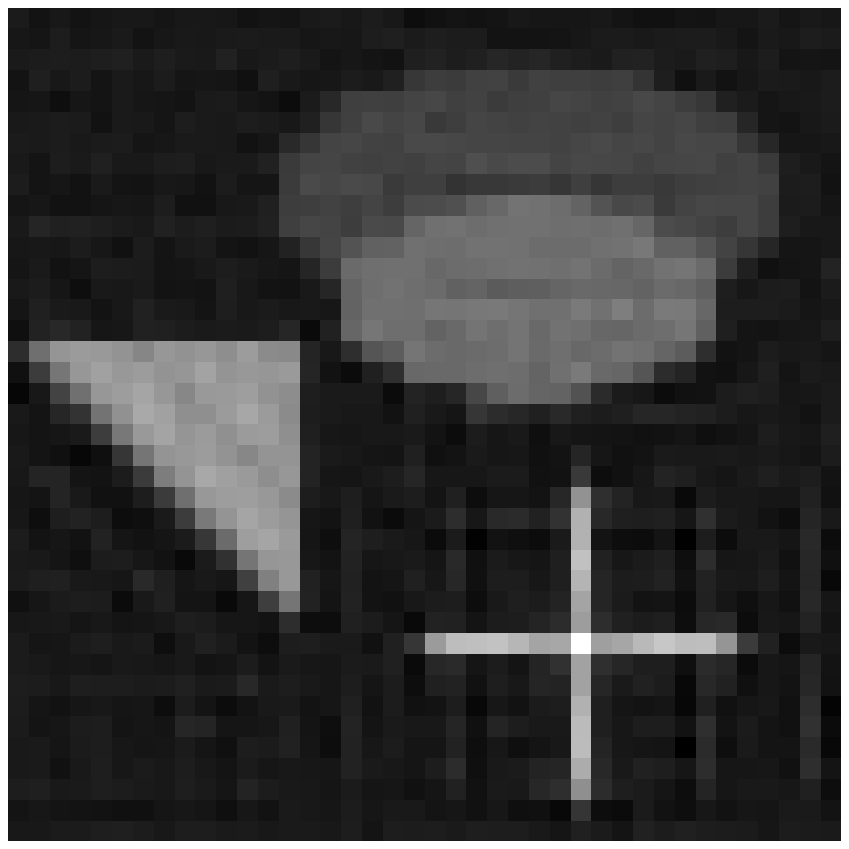}\\
(a) & (b) & (c)
\end{tabular}
\caption{Example 3 -- ``blur'' reconstructions: (a)  NSIFT and 
(b) NSIWT with parameters in \eqref{eq:Ns_succession}, (c) NSIT with $\alpha_n=0.01\cdot 0.7^n$.}
\label{fig:blur_rec}
\end{figure}
\clearpage
%-----------------------------------------------------------------
\section*{Acknowledgement}
%-----------------------------------------------------------------
The authors  warmly thank L. Reichel for illuminating discussions. 
This work is partly supported by PRIN 2012 N. 2012MTE38N for the first three authors, while the work
of the fourth author is partly supported by the Program `Becoming the Number One -- Sweden (2014)' of the Knut and Alice Wallenberg Foundation.

\bibliography{bibliografia_numerica}
\bibliographystyle{plain}

%-----------------------------------------------------------------
\section*{Appendix A}
%-----------------------------------------------------------------

\textbf{Lemma \ref{lem:6.3}}

\begin{proof}
Obviously, both the series converge or diverge simultaneously due to the Asymptotic Comparison test. If they converge, the thesis follows trivially. On the contrary, if they both diverge then we conclude by observing that $\sum_{k=N}^n t_k / \sum_{k=1}^n t_k$ is a monotonic increasing sequence bounded from above by $1$. Indeed, if we set
$$
A_n := \sum_{k=N}^n t_k, \qquad   B_n := \sum_{k=1}^n t_k,
$$
for every $n \geq N$ and for every $x \geq 0$ the function
$$
h_n (x) = \frac{A_n + x}{B_n + x}
$$
is monotone increasing with $h_n(x) \leq 1$. Then $A_{n+1}/B_{n+1} \geq A_n/ B_n$ for every $n$ and it is easy to see that $\sup_n \{ A_n / B_n\} =1$.
\end{proof}

\noindent
\textbf{Lemma \ref{lem:6.2}}
\begin{proof}
If $\lim_{k\to \infty} t_k = t \in (0, \infty]$, then
\begin{equation}\label{eq:eqtk}
\frac{1}{t_k} \sim \left(\frac{1}{\lambda} + \frac{1}{t} \right) \frac{1}{1 + t_k},
\end{equation}
where $1/t=0$ if $t=\infty$. Therefore, from the Asymptotic Comparison test for series, both series converge or diverge simultaneously. When they converge the thesis follows trivially. Assume then that the series diverge. 
Without loss of generality and for the sake of clarity we will prove the statement for $\lambda=1$. 
If we set 
$$
X_n := \frac{ \sum_{k=1}^n \frac{1}{t_k}}{\sum_{k=1}^n \frac{1}{1+ t_k}},
$$
we want to show that the limit of $X_n$ exists finite and, moreover, that is $\lim_{n\to \infty} X_n = 1 + 1/t$. Indeed, for any fixed $\epsilon >0$ there exists $N^1 _\epsilon$ such that for any $k \geq N^1 _\epsilon$ it holds that
\begin{equation}\label{eq:app1}
\frac{1}{t_k} < \left( 1 + \frac{1}{t} + \frac{\epsilon}{2} \right) \frac{1}{1 + t_k},
\end{equation}
and for any fixed $\epsilon$ and $N^1 _\epsilon$, there exists $N^2 _\epsilon$ such that for every $n \geq N^2 _\epsilon$ it holds that
\begin{equation}\label{eq:app2}
\frac{ \sum_{k=1}^{N^1 _\epsilon} \frac{1}{t_k}}{\sum_{k=1}^{n} \frac{1}{1+ t_k}} < \frac{\epsilon}{2}.
\end{equation}
Hence, for any $n \geq \max \{N^1 _\epsilon, N^2 _\epsilon  \}$, thanks to \eqref{eq:app1} and \eqref{eq:app2}, we have that
$$
X_{n} = \frac{\sum_{k=1}^{n} \frac{1}{t_k}}{\sum_{k=1}^{n} \frac{1}{1+ t_k}} < \frac{ \sum_{k=1}^{N^1 _\epsilon} \frac{1}{t_k}}{\sum_{k=1}^{n} \frac{1}{1+ t_k}} + \left( 1 + \frac{1}{t} + \frac{\epsilon}{2}  \right) \frac{ \sum_{k=N^1 _\epsilon +1}^{n} \frac{1}{1 + t_k} }{\sum_{k=1}^{n} \frac{1}{1+ t_k}}  < \frac{\epsilon}{2} + 1 + \frac{1}{t} + \frac{\epsilon}{2} = 1 + \frac{1}{t} + \epsilon.
$$
On the other hand, there exists $N^3_\epsilon$ such that for every $k \geq N^3 _\epsilon$ it holds
\begin{equation}\label{eq:app3}
\frac{1}{t_k} > \left( 1 + \frac{1}{t} - \frac{\epsilon}{2} \right) \frac{1}{1 + t_k},
\end{equation}
and, by Lemma \ref{lem:6.3}, for any fixed $N^3_\epsilon$ and for any fixed $\delta < \frac{\epsilon}{2}(1+ \frac{1}{t} - \frac{\epsilon}{2})^{-1}$, there exists $N^4_\epsilon$ such that for every $n \geq N^4_\epsilon$ it holds
\begin{equation}\label{eq:app4}
\frac{ \sum_{k=N^3 _\epsilon +1}^n \frac{1}{1 + t_k}}{\sum_{k=1}^{n} \frac{1}{1+ t_k}} > \left(1 -\delta \right).
\end{equation}
Hence, fo any $n \geq \max \{N^3_\epsilon, N^4_\epsilon  \}$, thanks to \eqref{eq:app3} and \eqref{eq:app4}, we have that
$$
X_{n} = \frac{\sum_{k=1}^{n} \frac{1}{t_k}}{\sum_{k=1}^{n} \frac{1}{1+ t_k}} > \frac{ \sum_{k=1}^{N^1 _\epsilon} \frac{1}{t_k}}{\sum_{k=1}^{n} \frac{1}{1+ t_k}} + \left( 1 + \frac{1}{t} - \frac{\epsilon}{2}  \right) \frac{ \sum_{k=N^1 _\epsilon +1}^{n} \frac{1}{1 + t_k} }{\sum_{k=1}^{n} \frac{1}{1+ t_k}}  >  \left(1 + \frac{1}{t} - \frac{\epsilon}{2}\right) (1-\delta) >  1 + \frac{1}{t} - \epsilon.
$$
Then, choosing $n \geq \max \{N^i _\epsilon : i=1, 2, 3,4 \}$,  the proof is concluded.
\end{proof}

\end{document}